\documentclass[reqno]{amsart}
\usepackage[
left=1.25in, right=1.25in]{geometry}
\usepackage{tikz}
\usetikzlibrary{calc}
\usepackage[all]{xy}
\usepackage{graphicx}
\usepackage{mathrsfs}
\usepackage{hyperref}
\usepackage{microtype}
\usepackage{braket}

\usetikzlibrary{decorations.markings}

\newcommand{\ba}{\begin{align}}
\newcommand{\ea}{\end{align}}
\newcommand{\be}{\begin{equation}}
\newcommand{\ee}{\end{equation}}

\newcommand{\beq}{\begin{eqnarray}}
\newcommand{\eeq}{\end{eqnarray}}

\newcommand{\beqs}{\begin{eqnarray*}}
\newcommand{\eeqs}{\end{eqnarray*}}

\newcommand\Supp{\mathcal{S}}
\newcommand\PA{\mathscr{P}}
\newcommand\SA{\mathscr{S}}

\newcommand\FS{\mathfrak{F}_S}
\newcommand\C{\mathscr{C}}
\newcommand\Z{\mathbb{Z}}

\newcommand\Max{|Max\rangle}
\newcommand\GHZ{|GHZ\rangle}

\begin{document}

\tikzset{->-/.style={decoration={
  markings,
  mark=at position #1 with {\arrow{>}}},postaction={decorate}}}

\theoremstyle{plain}
\newtheorem{theorem}{Theorem~}[section]
\newtheorem{main}{Main Theorem~}
\newtheorem{lemma}[theorem]{Lemma~}
\newtheorem{proposition}[theorem]{Proposition~}
\newtheorem{corollary}[theorem]{Corollary~}
\newtheorem{definition}[theorem]{Definition~}
\newtheorem{notation}[theorem]{Notation~}
\newtheorem{example}[theorem]{Example~}
\newtheorem*{remark}{Remark~}
\newtheorem*{question}{Question}
\newtheorem*{claim}{Claim}
\newtheorem*{ac}{Acknowledgement}
\newtheorem*{conjecture}{Conjecture~}
\renewcommand{\proofname}{\bf Proof}



\newcommand{\rotateRPY}[3]
{   \pgfmathsetmacro{\rollangle}{#1}
    \pgfmathsetmacro{\pitchangle}{#2}
    \pgfmathsetmacro{\yawangle}{#3}

    \pgfmathsetmacro{\newxx}{cos(\yawangle)*cos(\pitchangle)}
    \pgfmathsetmacro{\newxy}{sin(\yawangle)*cos(\pitchangle)}
    \pgfmathsetmacro{\newxz}{-sin(\pitchangle)}
    \path (\newxx,\newxy,\newxz);
    \pgfgetlastxy{\nxx}{\nxy};

    \pgfmathsetmacro{\newyx}{cos(\yawangle)*sin(\pitchangle)*sin(\rollangle)-sin(\yawangle)*cos(\rollangle)}
    \pgfmathsetmacro{\newyy}{sin(\yawangle)*sin(\pitchangle)*sin(\rollangle)+ cos(\yawangle)*cos(\rollangle)}
    \pgfmathsetmacro{\newyz}{cos(\pitchangle)*sin(\rollangle)}
    \path (\newyx,\newyy,\newyz);
    \pgfgetlastxy{\nyx}{\nyy};

    \pgfmathsetmacro{\newzx}{cos(\yawangle)*sin(\pitchangle)*cos(\rollangle)+ sin(\yawangle)*sin(\rollangle)}
    \pgfmathsetmacro{\newzy}{sin(\yawangle)*sin(\pitchangle)*cos(\rollangle)-cos(\yawangle)*sin(\rollangle)}
    \pgfmathsetmacro{\newzz}{cos(\pitchangle)*cos(\rollangle)}
    \path (\newzx,\newzy,\newzz);
    \pgfgetlastxy{\nzx}{\nzy};
}

\newcommand{\myGT}[2]
{
    \pgftransformcm{2}{0}{0}{3}{\pgfpoint{#1cm}{#2cm}}
}



\tikzstyle WL=[line width=3pt,opacity=1.0]

\newcommand{\drawWL}[3]
{
    \draw[white,WL]  (#2) -- (#3);
    \draw[#1] (#2) -- (#3);
}



\newpage
\title{Quon language: surface algebras and Fourier duality
}
\author{Zhengwei Liu}
\date{\today}

\begin{abstract}
Quon language is a 3D picture language that we can apply to simulate mathematical concepts.
We introduce the surface algebras as an extension of the notion of planar algebras to higher genus surface.
We prove that there is a unique one-parameter extension.
The 2D defects on the surfaces are quons, and surface tangles are transformations.
We use quon language to simulate graphic states that appear in quantum information, and to simulate interesting quantities in modular tensor categories.
This simulation relates the pictorial Fourier duality of surface tangles and the algebraic Fourier duality induced by the S matrix of the modular tensor category. The pictorial Fourier duality also coincides with the graphic duality on the sphere.
For each pair of dual graphs, we obtain an algebraic identity related to the $S$ matrix.
These identities include well-known ones, such as the Verlinde formula; partially known ones, such as the 6j-symbol self-duality; and completely new ones.

\end{abstract}

\maketitle

\section{Introduction}


Quon language is a 3D picture language that we can apply to simulate mathematical concepts \cite{JLW-Quon}.
It was designed to answer a question in quantum information, where the underling symmetry is the group $\Z_2$ for qubits and $\Z_d$ for qudits.
\textcolor{black}{One can consider the quon language as a topological quantum field theory (TQFT) in the 3D space with lower dimensional defects, and a quon as a 2D defect on the boundary of the 3D TQFT.}
The underlying symmetry of the 3D picture language can be generalized to more general quantum symmetries captured by subfactor theory \cite{JonSun97, EvaKaw98}.


Jones introduced subfactor planar algebras as a topological axiomatization of the standard invariants of subfactors \cite{JonPA}.
One can consider a planar algebra as a 2D topological quantum field theory (TQFT) on the plane with line defects. A subfactor planar algebra is always spherical, so the theory also extends to a sphere.

A vector in the planar algebra of a subfactor is a morphism in the bi-module category associated with the subfactor.
From this point of view, a morphism is usually represented as a disc with $m$ boundary points on the top and $n$ boundary points at the bottom, and considered as a transformation with $m$ inputs and $n$ outputs,

In the 3D quon language, we consider these morphisms in planar algebras as quons, and consider planar tangles as transformations. This interpretation is similar to the original definition of Jones, which turns out to be more compatible with the notions in quantum information.
 A planar tangle has multiple input discs and one output disc. So it represents a transformation from multiple quons to one quon.

In quantum information, we usually consider multiple qubits and their transformations.
To simulate multiple quon transformations, we generalize planar tangles to spherical tangles with multiple input discs and output discs.  When we compose such tangles, we will obtain higher genus-surfaces. So we further extend the notion of planar algebras to higher genus surfaces, that we call surface algebras. The theory of planar algebras becomes the local theory of surface algebras.

There is a freedom to define the partition function of a sphere in this extension, denoted by $\zeta$. We show that the partition function of the genus-$g$ surface is $\zeta^{1-g}$, which detects the topological non-triviality. We prove that any non-degenerate spherical planar algebra has a unique extension of to a surface algebra for any non-zero $\zeta$. Therefore a subfactor not only defines a spherical planar algebra, but also a surface algebra parameterized by $\zeta$. The fruitful theory of subfactors provides many interesting examples.

In this paper, we take the subfactor to be the quantum double of a unitary modular tensor category $\C$, also known as the Drinfeld double \cite{Dri86}.
Then the 2-box space of the planar algebra of the subfactor is isomorphic to $L^2(Irr)$, where $Irr$ denotes the set of irreducible objects of $\C$.
Xu and the author proved that the associated subfactor planar algebra is unshaded \cite{LiuXu}.
Thus the 2-box space becomes the 4-box space of the unshaded planar algebra, denoted by $\SA_4$.
The unshaded condition is crucial to define the string Fourier transform (SFT) on one space.
Moreover, we proved that the SFT on $\SA_4$ is identical to the modular $S$ transformation of the MTC $\C$.
Both transformations have been considered as a generalization of the Fourier transform from different point of views. This identification relates the two different Fourier dualities for MTC and subfactors.
We restrict the 1-quon space as $\SA_4$, in order to study this pair of Fourier dualities. We list the correspondence for 1-quons in Fig.~\ref{Table:Fourier duality on 1-quons}; see \S \ref{Sec:Fourier duality on 1-quons} for details.

\begin{table}[h]
\caption{Fourier duality on MTC and 1-Quons}\label{Table:Fourier duality on 1-quons}
\begin{tabular}{|c|c|}
 \hline
   MTC & Quon \\ \hline
simple objects (Irr) & ortho-normal-basis \\ \hline
multiplication & multiplication  \\ \hline
fusion & convolution  \\ \hline
$S$ matrix & SFT $\FS$ \\ \hline
full subcategories $\C_K$ & biprojections $P_K$ \\ \hline
M\"{u}ger's center $\C_{\hat{K}}$ & $P_{\hat{K}}$ \\ \hline
${\hat{\hat{K}}}=K$ & $\FS^2(P_K)=P_K$ \\ \hline
$\dim \C_{\hat{K}}$ & $\Supp(P_K)$ \\ \hline
$\dim \C_K \dim \C_{\hat{K}} = \dim C $ & $\Supp(P_K)\Supp(P_{\hat{K}})=\delta^2$ \\ \hline
\end{tabular}
\end{table}

Verlinde proposed that the $S$ matrix diagonalizes the fusion in the framework of CFT, known as the Verlinde formula \cite{Ver88}.
The Fourier duality on 1-quons gives a conceptual proof addressing Verlinde's original observation given by the Lines 3-5 in Fig.~\ref{Table:Fourier duality on 1-quons}.

Jiang, Wu, and the author studied the Fourier analysis of the SFT on subfactors in \cite{JLW16}. Through this identification, we obtain many inequalities for the $S$ matrix, which will be discussed in a coming paper.
It is particularly interesting that the $\infty$-$1$ Hausdorff-Young inequality for SFT gives an important inequality for the $S$ matrix in unitary MTC proved by Terry Gannon \cite{Gan05}:
\begin{align}
\|\FS(\beta_Y)\|_1 &\leq \delta^{-1} \|\beta_Y\|_{\infty} \\
\Rightarrow \left|\frac{S_X^Y}{S_{X}^0}\right| &\leq \frac{S_{0}^Y}{S_{0}^0}.
\end{align}
The equality condition has been used by M\"{u}ger to define the center of full subcategories in $\C$, known as M\"{u}ger's center \cite{Mug03}.
On the other hands, Bisch and Jones introduced biprojections for subfactors and planar algebras by studying intermediate subfactors \cite{Bis94,BisJon97}.
The Lines 6-10 in Fig.~\ref{Table:Fourier duality on 1-quons} identifies full subcategories in $\C$ with biprojections in $\SA_{4}$, and the several corresponding results between M\"{u}ger's center and biprojections.

We extend the correspondence between $\FS$ and $S$ from 1-quons to $n$-quons using surface algebras, see Theorem \ref{Thm:Fourier duality}:

\begin{center}
 \begin{tikzpicture}
  \begin{scope}[node distance=4cm, auto, xscale=1,yscale=1]
  \foreach \x in {0,1,2,3} {
  \foreach \y in {0,1,2,3} {
  \coordinate (A\x\y) at ({2*\x},{.7*\y});
  }}
  \foreach \y in {0,3}{
  \node at (A0\y) {surface tangles};
  \node at (A3\y) {graphic quons};
  \draw[->] (A1\y)  to node {$Z$} (A2\y);
  }
  \draw[->] (A02) to node [swap] {$\vec{\FS}$} (A01);
  \draw[->] (A32) to node [swap] {$\vec{S}$} (A31);
  \end{scope}
  \end{tikzpicture}.
\end{center}

The left side is pictorial and $\FS$ could be considered as a global $90^{\circ}$ rotation. The right side is algebraic and the $S$-matrix is generalization of the discrete Fourier transform. The partition function $Z$ is a functor relating the pictorial Fourier duality and the algebraic Fourier duality.
%
%

In particular, the algebraic Fourier duality between the two qudit resource state $\GHZ$ and $\Max$ in quantum information turns out to be a pictorial Fourier duality in quon language \cite{JLW-Quon}, see \S \ref{Sec: GHZ Max} for details.
\be
Max_{n,g}= \vec{\FS}(GHZ_{n,g})
\quad \Rightarrow \quad
\Max_{n,g}= \vec{S}\GHZ_{n,g}
\ee
Now this result also apply to unitary MTCs.
Comparing the coefficients, we obtain the generalized Verlinde formula:
\be \label{Equ:1}
\Max_{n,g}= \vec{S}\GHZ_{n,g}
\quad \Rightarrow \quad
\dim(\vec{X},g)= \sum_{X\in Irr} (\prod_{i=1}^n S_{X_i}^{X}) (S_{X}^1)^{2-n-2g}
\ee
The generalized Verlinde formula was first proved by Moore and Seiberg in CFT \cite{MooSei89}. Here we prove it for any unitary MTC and any genus $g$.
We refer the readers to an interesting discussion about various versions of Verlinde formula on MathOverflow: \url{https://mathoverflow.net/questions/151221/verlindes-formula}.

Moreover, for each oriented graph $\Gamma$ on the sphere, we define a surface tangles as a fat graph of $\Gamma$. Then its SFT becomes a fat graph of $\hat{\Gamma}$, where $\hat{\Gamma}$ is the dual graph of $\Gamma$ forgetting the orientation.
So the pictorial Fourier duality also coincide with the graphical duality.
Therefore we obtain one algebraic identity as algebraic Fourier duality of quons from any graph $\Gamma$.
We give some examples including well known ones, such as the Verlinde formula mentioned above; partially known ones; and completely new ones.

If the graph $\Gamma$ is the tetrahedron, then the graphic self-duality of the tetrahedron gives an algebraic $6j$-symbol self-duality for unitary MTCs,  see \S \ref{Sec:Fourier duality} for details:
\be
\left|{{{X_{6}~X_{5}~X_{4}}\choose{\overline{X_{3}}~\overline{X_{2}}~\overline{X_{1}}}}}\right|^{2}
= \sum_{\vec {Y}\in Irr^6} \left(\prod_{k=1}^{6}S_{X_{k}}^{Y_{k}} \right)
\left|{{{Y_{1}Y_{2}Y_{3}}\choose{Y_{4}Y_{5}Y_{6}}}}\right|^{2}.
\ee
In the special case of quantum $SU(2)$, the identity for the 6j-symbol self-duality was discovered by Barrett in the study of quantum gravity \cite{Bar03}, based on an interesting identity of J. Robert \cite{Rob95}. Then the identity was generalized to some other cases related to $SU(2)$ in \cite{FNR07}. A general case for MTCs has been conjectured by Shamil Shakirov, which we answer positively here.

We obtain a sequence of new algebraic self-dual identities from a sequence of self-dual graphs, see \S \ref{Sec:Fourier duality} for details:
\be
\left|{{{X_{2n}~X_{2n-1}\cdots X_{n}}\choose{\overline{X_{n}} \phantom{aa} \overline{X_{n-1}} ~\cdots \overline{X_{1}}}}}\right|^{2}
= \sum_{\vec {Y}\in Irr^{2n}} \left(\prod_{k=1}^{2n}S_{X_{k}}^{Y_{k}} \right)
\left|{{{Y_{1} \phantom{aa} Y_{2} \phantom{aa} \cdots ~Y_{n}}\choose{Y_{n+1}Y_{n+2}\cdots Y_{2n}}}}\right|^{2}.
\ee

\begin{ac}
The author would like to thank Terry Gannon, Arthur Jaffe, Vaughan F. R. Jones, Shamil Shakirov, Cumrun Vafa, Erik Verlinde, Jinsong Wu and Feng Xu for helpful discussions. The author was supported by a grant from Templeton Religion Trust and an AMS-Simons Travel Grant.
The author would like to thank the Isaac Newton Institute for Mathematical Sciences, Cambridge, for support and hospitality during the programme ``Operator algebras: subfactors and their applications''.
\end{ac}

\section{Surface algebras}
In this section, we are going to extend spherical planar algebras from the sphere to higher genus surfaces, which are the boundary of 3-manifolds in the 3D space. The theory of spherical planar algebras become the 0-genus case.

To simply the notation, we only define the single color case and the ground field is $\mathbb{C}$. One can generalize these definitions to multi-color cases over a general field.

\subsection{Surface tangles}

If we consider a planar tangle as a spherical tangle by one point compactification of the plane, then the complement of the planar tangle becomes a disc on the sphere. The induced orientation of the boundary of the output disc will be changed.
Thus we use anti-clockwise and clockwise orientations of boundary of discs to indicate input and output respectively.

The composition of planar tangles is still a planar tangle. In this case, the number of output disc is always one.
If we allow spherical tangles to have multiple input discs and output discs, then we will obtain tangles on higher genus surfaces when we compose these spherical tangles.
We give a generalization of planar tangles to surface tangles, see Fig.~\ref{Fig:genus-2 tangle} for example.

\begin{definition}
A genus-$g$ tangle, for $g\in \mathbb{N}$, is a 3-manifold in the 3D space whose boundary is a genus-$g$ surface. The surface consists of a finite (possibly empty) set of smooth closed discs $\mathcal{D}(T)$. For each disc $D\in \mathcal{D}(T)$, its boundary $\partial D$ of  is an oriented circle with a number of marked points. There is also a finite set of disjoint smoothly embedded curves called strings, which are either closed curves, or the end points of the strings are different marked points of discs. Each marked point is an end-point of some string, which meets the boundary of the corresponding disc transversally.

The connected components of the complement of the strings and discs are called regions. The connected component of the boundary of a disc, minus its marked points, will be called the intervals of that disc. To each disc there is a distinguished interval on its boundary. The distinguished interval is marked by an arrow $\to$, which also indicates the orientation.

A surface tangle is a disjoint union of finitely many higher-genus tangles.

\end{definition}

\begin{figure}\label{Fig:genus-2 tangle}
\begin{center}
\begin{tikzpicture}
\draw[blue] (-2,-1)--++(4,0) arc (-90:90:1) --++(-4,0) arc (90:270:1);
\draw[blue] (-1-.2,0) to [bend left=30] (-1+.2,0);
\draw[blue] (-1-.3,0) to [bend left=-30] (-1+.3,0);
\draw[blue] (1-.2,0) to [bend left=30] (1+.2,0);
\draw[blue] (1-.3,0) to [bend left=-30] (1+.3,0);

\draw (-2,0) to [bend left=30] (0,0);
\draw (-2,0) to [bend left=-30] (0,0);
\draw (2,0) to [bend left=30] (0,0);
\draw (2,0) to [bend left=-30] (0,0);

\draw (-2.5,0) to [bend left=30] (2.5,0);
\draw (-2.5,0) to [bend left=-30] (2.5,0);

\foreach \x in {-2,0,2} {
\fill[white] (\x,0) circle (.5);
\fill[blue!20] (\x,0) circle (.5);
}
\draw[blue,->] (-1.5,0) arc (0:360:.5);
\draw[blue,<-] (-.5,0) arc (-180:180:.5);
\draw[blue,->] (2.5,0) arc (0:360:.5);
\end{tikzpicture}
\end{center}\caption{Example: a genus-2 tangle with two input discs and one output disc.}
\end{figure}
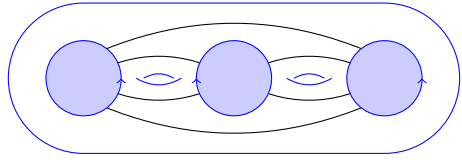

We consider the 3D topological isotopy by orientation-preserving diffeomorphisms in the 3D space.
The {\it surface operad} is the set of isotopic classes of surface tangles.

\begin{remark}
One can impose additional data to color the regions and the strings. In subfactor theory, an alternating shading of the regions is preferred. Therefore the number of boundary points of each disc is even. In tensor categories, the strings are colored by simple objects.  In 2-categories, one has multiple colors for regions and strings. In these cases, the boundary condition $\partial D$ will be colored too.
\end{remark}

\begin{notation}
Let $\partial \mathcal{D}$ be the set of boundary conditions of discs, i.e., the equivalent classes of $\partial D$ modulo isotopy.
We say a disc $D$ is an input (respectively, output) disc, if the orientation of $\partial D$ is anti-clockwise (respectively, clockwise). Let $\mathcal{D}_{I}$ and $\mathcal{D}_{O}$ be the sets of input discs and output discs respectively.
\end{notation}

\begin{notation}
Let $r$ be a reflection by a plane in the 3D space.
\end{notation}
The reflection $r$ is unique up to topological isotopy in the 3D space. Moreover $r$ maps a surface tangle to a surface tangle and reverses the orientation of the boundary of discs. Thus $r$ switches $\mathcal{D}_{I}$ and $\mathcal{D}_{O}$.

\begin{definition}
We define two elementary operadic operations for surface tangles.
\begin{itemize}
\item[(1)] Tensor: taking a disjoint union of two surface tangles.
\item[(2)] Contraction: gluing two discs of a surface tangle whose boundaries are mirror images.
\end{itemize}
\end{definition}
Modulo topological isotopy in the 3D space, the tensor is unique, but
there are inequivalent contractions.
The composition of planar tangles can be decomposed as a contraction and a tensor.

\subsection{Surface algebras}

We define surface algebras as finite dimensional representations of surface tangles whose target spaces are indexed by the boundary condition $\partial \mathcal{D}$:

\begin{definition}\label{Def:SA}
A surface algebra $\SA_{\bullet}$ is a representation $Z$ of surface tangles on the tensor products of a family of  finite dimensional vector spaces $\{\mathscr{S}_{i}\}_{i \in \partial \mathcal{D}}$, having the following axioms:
	\begin{itemize}
	\item[(1)] Boundary condition: For a surface tangle $T$, $Z(T)$ is a vector in $\displaystyle \bigotimes_{D\in \mathcal{D}(T)} \SA_{\partial D}$.
	\item[(1')] Second boundary condition: If $T$ has no discs, then $Z(T)$ is a scalar in the ground field.
	\item[(2)] Duality: For any $i \in \partial \mathcal{D}$, $\SA_{r(i)}$ is the dual space of $\SA_{i}$.
	\item[(3)] Isotopy invariance: The representation $Z$ is well-defined up to isotopy in the 3D space.
	\item[(4)] Naturality: The following commutative diagram holds:

\begin{center}
 \begin{tikzpicture}
  \begin{scope}[node distance=4cm, auto, xscale=1, yscale=1]
  \foreach \x in {0,1,2,3} {
  \foreach \y in {0,1,2,3} {
  \coordinate (A\x\y) at ({2*\x},{.7*\y});
  }}
  \foreach \y in {0,3}{
  \node at (A0\y) {surface tangles};
  \node at (A3\y) {vectors};
  \draw[->] (A1\y)  to node {$Z$} (A2\y);
  }
  \foreach \x in {0,3}{
  \draw[->] (A\x2) to node [swap] {tensor/contraction} (A\x1);
  }
  \end{scope}
  \end{tikzpicture}
\end{center}

	\end{itemize}

\end{definition}
We also call $Z(T)$ the {\it partition function} of $T$ from the statistic point of view.

\begin{definition}
The partition function of a sphere is called the 2D sphere value, denoted by $\zeta$.
The partition function of a closed string in a sphere is $\delta\zeta$. We call $\delta$ the 1D circle value.
\end{definition}

If we restrict the representation $Z$ to genus-0 tangles with one output disc, then we recover unital, finite dimensional, spherical planar algebras \footnote{
The spherical condition for planar algebra is defined based on the evaluable condition, namely the 0-box space is one-dimensional \cite{Jon12}. The spherical condition of surface algebras on the sphere does not require this one-dimensional condition. Typical examples of such planar algebras are graph planar algebras. \cite{Jon00}.}. Moreover, $\delta$ is the statistical dimension of the planar algebra.

\begin{definition}
We say a surface algebra is an extension of a planar algebra, if the restriction of its partition function  $Z$ on the planar tangles is the partition function of the planar algebra.
\end{definition}

\begin{remark}
If the regions and strings of surface tangles are colored, then the index set $\mathbb{N}$ will be replaced by all permissible colors of the boundary of a disc.
\end{remark}

\begin{remark}
If one considers surface algebras as 2D TQFT with line defects, then it is better to consider the discs of surface tangles as holes. However, we emphasize that these surfaces are boundaries of 3-manifolds, thus the notion of discs is more reasonable.
\end{remark}

\begin{notation}
For an input disc $D$, the boundary condition $\partial D$ only depends on the number of marked points $n$. Thus we denote $\SA_{\partial D}$ by $\SA_n$ and its dual space by $\SA_n^*$.
\end{notation}

We can consider $Z(T)$ as a multi-linear transformation on the vector space  $\{\mathscr{S}_{n}\}_{n \in \mathbb{N}}$ from input discs to output discs.

Let us extend some notions from planar algebra to surface algebras.

%
%

\begin{notation}
We use a thick string labelled by a number $n$ to indicates $n$ parallel strings.
\end{notation}

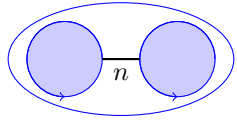
\begin{figure}[h]
\begin{tikzpicture}
\draw[blue] (.25,0) ellipse (1.5 and .75);
\draw[fill,blue!20] (0,0) arc (0:360:.5);
\draw[blue] (0,0) arc (0:360:.5);
\draw[->,blue] (0,0) arc (0:270:.5);
\draw[fill,blue!20] (1.5,0) arc (0:360:.5);
\draw[blue] (1.5,0) arc (0:360:.5);
\draw[->,blue] (1.5,0) arc (0:270:.5);
\draw[thick] (0,0)--(.5,0);
\node at (.25,-.2) {$n$};
\end{tikzpicture}
\caption{The genus-0 tangle for the bilinear form $B_n$.} \label{Fig:Bn}
\end{figure}

\begin{notation}
The genus-0 tangle in Fig.~\ref{Fig:Bn} defines a bilinear form $B_n$ on $\SA_{n}$.
\end{notation}

\begin{definition}
A surface algebra is called
{\it non-degenerate}, if the bilinear form $B_n$  is non-degenerate for all $n$.
\end{definition}

If the surface algebra is non-degenerate, then the bilinear form $B_n$ induces an isomorphism $D_n$ from the vector space $\SA_{n}$ to its dual space $\SA_{n}^*$. From this point of view, the tangles for the map $D_n$ and its inverse $D_n^{-1}$ are given in Fig.~\ref{Fig:Dn}. So we can identity the vector space $\SA_n$ with its dual using these duality maps, denoted by $D$ for short.

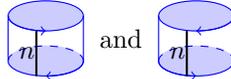
\begin{figure}
\begin{tikzpicture}
\begin{scope}[xscale=.25,yscale=.1]
\foreach \x in {0.5}{
\draw[fill,blue!20] (\x,0) arc (-180:180:2);
\draw[fill,blue!20] (\x,6) arc (-180:180:2);
\draw[blue] (\x,6) arc (-180:180:2);
\draw[->,blue] (\x,6) arc (-180:-90:2);
\draw[blue] (\x,0) arc (-180:0:2);
\draw[->,blue] (\x+4,0) arc (0:-90:2);
\draw[blue,dashed] (\x,0) arc (180:0:2);
}
\draw[blue] (.5,0)--(.5,6);
\draw[blue] (4.5,0)--(4.5,6);
\node at (6.5,3) {and};
\draw[thick] (2,-2)--(2,4);
\node at (2-.5,1) {$n$};
\end{scope}

\begin{scope}[shift={(2,0)},xscale=.25,yscale=.1]
\foreach \x in {0.5}{
\draw[fill,blue!20] (\x,0) arc (-180:180:2);
\draw[fill,blue!20] (\x,6) arc (-180:180:2);
\draw[blue] (\x,6) arc (-180:180:2);
\draw[->,blue] (\x+4,6) arc (0:-90:2);
\draw[blue] (\x,0) arc (-180:0:2);
\draw[->,blue] (\x,0) arc (-180:-90:2);
\draw[blue,dashed] (\x,0) arc (180:0:2);
}
\draw[blue] (.5,0)--(.5,6);
\draw[blue] (4.5,0)--(4.5,6);
\draw[thick] (2,-2)--(2,4);
\node at (2-.5,1) {$n$};
\end{scope}
\end{tikzpicture}
\caption{The tangles for $D_n$ and $D_n^{-1}$.}  \label{Fig:Dn}
\end{figure}

\begin{definition}
Suppose $^*$ is an anti-linear involution on $\SA_{n}$, $n\in \mathbb{N}$. Then $R(x):= D(x^*)$ is an anti-linear isomorphism from $\SA_{n}$ to $\SA_{n}^*$. We still denote its inverse and the linear extension on the tenor power by $R$. Then
\be
\langle x,y\rangle:=B_n(x,y^*) =R(y)(x)
\ee
is an inner product on $\SA_{n}$.
\end{definition}

\begin{remark}
The bilinear form in planar algebras is $\frac{1}{\zeta}B_n$.
\end{remark}

\begin{definition}

A surface algebra is called a surface $^*$-algebra, if it has an anti-linear involution, such that for any surface tangle $T$,
\be
Z(r(T))=R(Z(T)).
\ee

\end{definition}

\begin{definition}
A surface $^*$-algebra $\SA_{\bullet}$ is called (semi-)positive, if the inner product $\langle \cdot,\cdot \rangle$ is (semi-)positive.
\end{definition}

Note that positivity implies non-degeneracy.

For a positive surface algebra $\SA_{\bullet}$, the vector space $\SA_n$ is a Hilbert space. Moreover, the map $R$ is the Riesz representation. Thus we can consider a positive surface algebra as a Hilbert space representation of surface tangles satisfying an additional commutative diagram:

\be \label{Equ:reflection}
\raisebox{-1cm}{
 \begin{tikzpicture}
  \begin{scope}[node distance=4cm, auto, xscale=1, yscale=1]
  \foreach \x in {0,1,2,3} {
  \foreach \y in {0,1,2,3} {
  \coordinate (A\x\y) at ({2*\x},{.7*\y});
  }}
  \foreach \y in {0,3}{
  \node at (A0\y) {surface tangles};
  \node at (A3\y) {vectors};
  \draw[->] (A1\y)  to node {$Z$} (A2\y);
  }
  \draw[->] (A02) to node [swap] {reflection} (A01);
  \draw[->] (A32) to node [swap] {Riesz representation} (A31);
  \end{scope}
  \end{tikzpicture}} ~~.
\ee


\subsection{Labelled tangles}
For a surface tangle, we can partially fill its discs by a vector with compatible boundary condition. We consider the result as a labelled tangle. Let us extend the representation $Z$ of surface tangles to labelled tangles.

\begin{definition}
Suppose $\SA_{\bullet}$ is a surface algebra and $T$ is a surface tangle. Let $S$ be a subset of $\mathcal{D}(T)$ and $v$ be a vector in  $\displaystyle \bigotimes_{D\in S} \SA_{r(\partial D)}$.
We call the pair $T$ and $v$ a labelled tangle, denoted by $T \circ_S v$, or $T(v)$ for short, in the sense that the discs in $S$ are labelled by the vector $v$.
We call it fully labelled, if all discs are labelled.
We define the partition function of the labelled tangle $T(v)$ by
\be \label{Equ:ZTv}
Z(T(v)):=v(Z(T)),
\ee
where $\displaystyle v \in \bigotimes_{D\in S} \SA_{r(\partial D)}$ is considered as a partial linear functional on $\displaystyle \bigotimes_{D\in \mathcal{D}(T)} \SA_{\partial D}$.
\end{definition}

\begin{definition}
We define the reflection on a labelled tangle $T(v)$ by
\be \label{Equ:rTv}
r(T(v))=r(T)(R(v)).
\ee

\end{definition}

Suppose $\SA_{\bullet}$ is a surface algebra and $T$ is a surface tangle. Then $Z(T) \in \displaystyle \bigotimes_{D\in \mathcal{D}(T)} \SA_{\partial D}$.
Let $S$ be a subset of $\mathcal{D}(T)$. Then each vector $v$ in  $\displaystyle \bigotimes_{D\in S} \SA_{r(\partial D)}$ is a partial linear functional on $\displaystyle \bigotimes_{D\in \mathcal{D}(T)} \SA_{\partial D}$. Moreover, $v(Z(T))$ is a vector in $\displaystyle \bigotimes_{D\in \mathcal{D}(T)\setminus S} \SA_{r(\partial D)}$, corresponding to the unlabelled discs of $T$.

\begin{theorem}\label{Thm:Labelled tangles}
For a surface algebra $\SA_{\bullet}$, the extended representation $Z$ of labelled tangles satisfies all axioms in Definition \ref{Def:SA}.
\end{theorem}

\begin{proof}
Axioms (1') and (3) follow from the corresponding axioms for surface tangles.

Axiom (1) and (4) follow from the corresponding axioms for surface tangles and Equation \eqref{Equ:ZTv}.

Axiom (2) follows the corresponding axioms for surface tangles and Equation \eqref{Equ:rTv}.

We give a proof for the tensor in axiom (4) in details. The others are similar.
Suppose $T_1(v_1)$ and $T_2(v_2)$ are labelled tangles, then their disjoint union $T_1(v_1) \otimes T_2(v_2)=(T_1\otimes T_2)(v_1\otimes v_2)$
 is a labelled tangle. So
\begin{align*}
  &Z(T_1(v_1) \otimes T_2(v_2))\\
=&Z((T_1\otimes T_2)(v_1\otimes v_2))\\
=&(v_1\otimes v_2)(Z(T_1\otimes T_2))\\
=&(v_1\otimes v_2)(Z(T_1)\otimes Z(T_2))\\
=&v_1(Z(T_1)) \otimes v_2(Z(T_2))\\
=&Z(T_1(v_1)) \otimes Z(T_2(v_2))
\end{align*}
\end{proof}

Let $T(v)$ be a a labelled tangle containing $T_1(v_1)$ as a sub labelled tangle. In other words, there is a labelled tangle $T_2(v_2)$,
such that $T(v)$ is a multiple contractions between $T_1(v_1)$ and $T_2(v_2)$. We denote it by $T(v)= T_1(v_1) \circ_S T_2(v_2)$, where $S$ indicates the unlabelled discs are that glued.
If $T_3(v_3)$ is a labelled tangle which has the same partition function as $T_1(v_1)$. Then we can identify their unlabelled discs. If we replace $T_1(v_1)$
by $T_3(v_3)$ in $T(v)$, then we obtain a new labelled tangle  $T_3(v_3) \circ_S T_2(v_2)$. By Theorem \ref{Thm:Labelled tangles}, we have that

\begin{corollary}
If $Z(T_1(v_1))=Z(T_3(v_3))$, then $Z(T_1(v_1) \circ_S T_2(v_2))=Z(T_3(v_3) \circ_S T_2(v_2))$.
\end{corollary}

Since the replacement of $T_1(v_1)$ by $T_3(v_3)$ will not affect the partition function, we can consider it as a {\it relation} of labelled tangles, denoted by $T_1(v_1)=T_3(v_3)$.


The following genus-0 tangle $I_n$ has one input disc and one output disc:
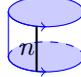
\begin{center}
\begin{tikzpicture}
\begin{scope}[xscale=.25,yscale=.1]
\foreach \x in {0.5}{
\draw[fill,blue!20] (\x,0) arc (-180:180:2);
\draw[fill,blue!20] (\x,6) arc (-180:180:2);
\draw[blue] (\x,6) arc (-180:180:2);
\draw[->,blue] (\x,6) arc (-180:-90:2);
\draw[blue] (\x,0) arc (-180:0:2);
\draw[->,blue] (\x,0) arc (-180:-90:2);
\draw[blue,dashed] (\x,0) arc (180:0:2);
}
\draw[blue] (.5,0)--(.5,6);
\draw[blue] (4.5,0)--(4.5,6);
\draw[thick] (2,-2)--(2,4);
\node at (2-.5,1) {$n$};
\end{scope}
\end{tikzpicture}
\end{center}
If $\SA_{\bullet}$ is non-degenerate, then $I_n$ defines the identity map on $\SA_n$.

For any vector $v$ in $\SA_n$, we obtain a labelled tangle  $I_n(v)$ by filling $v$ in the input disc. Then $Z(I_n(v))=v$.
So the vector $v$ can be considered as a labelled tangle, denoted by $v=I_n(v)$. Its pictorial representation is

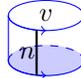
\begin{center}
\begin{tikzpicture}
\begin{scope}[xscale=.25,yscale=.1]
\foreach \x in {0.5}{
\draw[fill,blue!20] (\x,0) arc (-180:180:2);
\draw[blue] (\x,6) arc (-180:180:2);
\draw[->,blue] (\x,6) arc (-180:-90:2);
\draw[blue] (\x,0) arc (-180:0:2);
\draw[->,blue] (\x,0) arc (-180:-90:2);
\draw[blue,dashed] (\x,0) arc (180:0:2);
}
\draw[blue] (.5,0)--(.5,6);
\draw[blue] (4.5,0)--(4.5,6);
\draw[thick] (2,-2)--(2,4);
\node at (2-.5,1) {$n$};
\node at (2.5,6) {$v$};
\end{scope}
\end{tikzpicture}
\end{center}

This construction can be generalized to any vector in the tensor product of $\{\SA_{n}\}_{n\in\mathbb{N}}$ and their dual spaces.
Therefore, we can identify vectors and labelled tangles by each other in a surface algebra.

The vector spaces $\SA_{n}$ and $\SA_{n}^*$ are dual to each other.
Let $\{\alpha_k\}$ be a basis of $\SA_{n}$ and $\{\beta_k\}$ be its dual basis. Then we have that
\be
Z(I_n)=\sum_k \alpha_k\otimes \beta_k.
\ee
The right hand side is independent of the choice of basis.

This defines a relation for labelled tangles that we call the {\it joint relation}.
\begin{proposition}[Joint relation]\label{Prop:Joint relation}
Suppose $\SA_{\bullet}$ is a non-degenerate surface algebra, then we have the join relation for labelled tangels:
\be \label{Equ:joint relation}
\raisebox{-.5cm}{
\begin{tikzpicture}
\begin{scope}[xscale=.25,yscale=.1]
\foreach \x in {0.5}{
\draw[fill,blue!20] (\x,0) arc (-180:180:2);
\draw[fill,blue!20] (\x,6) arc (-180:180:2);
\draw[blue] (\x,6) arc (-180:180:2);
\draw[->,blue] (\x,6) arc (-180:-90:2);
\draw[blue] (\x,0) arc (-180:0:2);
\draw[->,blue] (\x,0) arc (-180:-90:2);
\draw[blue,dashed] (\x,0) arc (180:0:2);
}
\draw[blue] (.5,0)--(.5,6);
\draw[blue] (4.5,0)--(4.5,6);
\draw[thick] (2,-2)--(2,4);
\node at (2-.5,1) {$n$};
\end{scope}
\end{tikzpicture}
}
=\sum_k
\raisebox{-1cm}{
\begin{tikzpicture}
\begin{scope}[xscale=.25,yscale=.1]
\foreach \x in {0.5}{
\draw[fill,blue!20] (\x,0) arc (-180:180:2);
\draw[blue] (\x,6) arc (-180:180:2);
\draw[->,blue] (\x,6) arc (-180:-90:2);
\draw[blue] (\x,0) arc (-180:0:2);
\draw[->,blue] (\x,0) arc (-180:-90:2);
\draw[blue,dashed] (\x,0) arc (180:0:2);
}
\draw[blue] (.5,0)--(.5,6);
\draw[blue] (4.5,0)--(4.5,6);
\draw[thick] (2,-2)--(2,4);
\node at (2-.5,1) {$n$};
\node at (2.5,6) {$\alpha_k$};
\end{scope}

\begin{scope}[xscale=.25,yscale=.1,shift={(0,12)}]
\foreach \x in {0.5}{
\draw[fill,blue!20] (\x,6) arc (-180:180:2);
\draw[blue] (\x,6) arc (-180:180:2);
\draw[->,blue] (\x,6) arc (-180:-90:2);
\draw[blue] (\x,0) arc (-180:0:2);
\draw[->,blue] (\x,0) arc (-180:-90:2);
\draw[blue,dashed] (\x,0) arc (180:0:2);
}
\draw[blue] (.5,0)--(.5,6);
\draw[blue] (4.5,0)--(4.5,6);
\draw[thick] (2,-2)--(2,4);
\node at (2-.5,1) {$n$};
\node at (2.5,0) {$\beta_k$};
\end{scope}
\end{tikzpicture}}
\ee
Consequently, for the genus-$g$ surface $S_g$, we have
\be
Z(S_g)=\zeta^{1-g}.
\ee

\end{proposition}

\subsection{Unique extension}

\begin{theorem}\label{Thm:unique extension}
For any $\zeta\neq0$, any non-degenerate, unital, finite dimensional, spherical planar algebra $\PA_{\bullet}$ has an unique extension to a non-degenerate  surface algebra $\SA_{\bullet}$ with 2D sphere value $\zeta$.
\end{theorem}

In other words, the joint relation and the local relations defined by the planar algebra are consistent and the  2D sphere value $\zeta$ is a freedom.

\begin{proof}
Since $\PA_{\bullet}$ is non-degenerate and $\zeta \neq 0$, its extension $\SA_{\bullet}$ is non-degenerate. Moreover, the inner product on  $\SA_{\bullet}$ is $\zeta$ times the inner product on $\PA_{\bullet}$. The anti-linear isomorphism $D_n: \SA_n \to \SA_n^*$ is defined by the Riesz representation.

The interior 3-manifold of a fully labelled surface tangle $T$ is contractable to a graph $G_T$, homotopic to a planar graph.
Moreover, the graph $G_T$ is unique up to the contraction move which contracts an adjacent pair of an $m$-valent vertex and an $n$-valent vertex to an $(m+n-2)$-valent vertex:

\begin{center}
\begin{tikzpicture}
\begin{scope}[scale=.5]
\foreach \x in {0,1,2,3}
{
\coordinate (A\x) at ({cos(360/4*\x)}, {sin(360/4*\x)});
\draw (0,0)--(A\x);
}
\begin{scope}[shift={(-1,0)}]
\foreach \x in {0,1,2}
{
\coordinate (A\x) at ({cos(360/3*\x)}, {sin(360/3*\x)});
\draw (0,0)--(A\x);
}
\end{scope}
\node at (2,0) {$\to$};
\node at (6,0) {$.$};
\begin{scope}[shift={(4,0)}]
\foreach \x in {0,1,2,3,4}
{
\coordinate (A\x) at ({cos(360/5*\x)}, {sin(360/5*\x)});
\draw (0,0)--(A\x);
}
\end{scope}
\end{scope}
\end{tikzpicture}
\end{center}

We consider $T$ as a small neighborhood of $G_T$. We can decompose $T$ into fully labelled genus-0 tangles by applying the joint relation \eqref{Equ:joint relation} to all edges of $G_T$.   Thus the partition function $Z(T)$ is determined by the value of $Z$ on fully labelled genus-0 tangles. Therefore the extension is unique for a fixed $\zeta$.

Now we prove the existence of such extension.
We need to prove that the partition function  $Z(T)$ is well-defined.

Let $\{\alpha_k\}$ be a basis of $\SA_{n}$ and $\{\beta_k\}$ be its dual basis.
Let $\{\alpha_{k'}\}$ be a basis of $\SA_{m}$ and $\{\beta_{k'}\}$ be its dual basis.
By basic linear algebra, for any $f\in \SA_{n+m}$,  we have that

\be
\sum_k
\raisebox{-1.5cm}{
\begin{tikzpicture}
\begin{scope}[xscale=.25,yscale=.1]
\foreach \x in {0.5}{
\draw[fill,blue!20] (\x,0) arc (-180:180:2);
\draw[blue] (\x,6) arc (-180:180:2);
\draw[->,blue] (\x,6) arc (-180:-90:2);
\draw[blue] (\x,0) arc (-180:0:2);
\draw[->,blue] (\x,0) arc (-180:-90:2);
\draw[blue,dashed] (\x,0) arc (180:0:2);
}
\draw[blue] (.5,0)--(.5,6);
\draw[blue] (4.5,0)--(4.5,6);
\draw[thick] (2,-2)--(2,4);
\node at (2-.5,1) {$n$};
\node at (2.5,6) {$\alpha_k$};
\end{scope}

\begin{scope}[xscale=.25,yscale=.1,shift={(0,12)}]
\foreach \x in {0.5}{
\draw[fill,blue!20] (\x,12) arc (-180:180:2);
\draw[blue] (\x,12) arc (-180:180:2);
\draw[->,blue] (\x,12) arc (-180:-90:2);
\draw[blue] (\x,0) arc (-180:0:2);
\draw[->,blue] (\x,0) arc (-180:-90:2);
\draw[blue,dashed] (\x,0) arc (180:0:2);
}
\draw[blue] (.5,0)--(.5,12);
\draw[blue] (4.5,0)--(4.5,12);
\draw[thick] (2,-2)--(2,10);
\draw[fill,white] (4,5) arc (0:360:1.5 and 2);
\draw[->,blue]  (4,5) arc (0:360:1.5 and 2);
\node at (2.5,5) {$f$};
\node at (2-.5,8) {$m$};
\node at (2-.5,1) {$n$};
\node at (2.5,0) {$\beta_k$};
\end{scope}
\end{tikzpicture}}
=\sum_{k'}
\raisebox{-1.5cm}{
\begin{tikzpicture}
\begin{scope}[xscale=.25,yscale=.1]
\foreach \x in {0.5}{
\draw[fill,blue!20] (\x,-6) arc (-180:180:2);
\draw[blue] (\x,6) arc (-180:180:2);
\draw[->,blue] (\x,6) arc (-180:-90:2);
\draw[blue] (\x,-6) arc (-180:0:2);
\draw[->,blue] (\x,-6) arc (-180:-90:2);
\draw[blue,dashed] (\x,-6) arc (180:0:2);
}
\draw[blue] (.5,-6)--(.5,6);
\draw[blue] (4.5,-6)--(4.5,6);
\draw[thick] (2,-8)--(2,4);
\draw[fill,white] (4,-2) arc (0:360:1.5 and 2);
\draw[->,blue]  (4,-2) arc (0:360:1.5 and 2);
\node at (2.5,-2) {$f$};
\node at (2-.5,1) {$m$};
\node at (2-.5,-5) {$n$};
\node at (2.5,6) {$\alpha'_k$};
\end{scope}

\begin{scope}[xscale=.25,yscale=.1,shift={(0,12)}]
\foreach \x in {0.5}{
\draw[fill,blue!20] (\x,6) arc (-180:180:2);
\draw[blue] (\x,6) arc (-180:180:2);
\draw[->,blue] (\x,6) arc (-180:-90:2);
\draw[blue] (\x,0) arc (-180:0:2);
\draw[->,blue] (\x,0) arc (-180:-90:2);
\draw[blue,dashed] (\x,0) arc (180:0:2);
}
\draw[blue] (.5,0)--(.5,6);
\draw[blue] (4.5,0)--(4.5,6);
\draw[thick] (2,-2)--(2,4);
\node at (2-.5,1) {$m$};
\node at (2.5,0) {$\beta'_k$};
\end{scope}
\end{tikzpicture}}
\;.
\ee
Therefore for a fixed $G_T$, $Z(T)$ is well-defined up to isotopy.

By basic linear algebra, for any $\alpha \in \PA_{n}$ and $\beta \in \PA_{n}^*$,  we have that

\be \label{Equ:joint relation}
\raisebox{-.5cm}{
\begin{tikzpicture}
\begin{scope}[xscale=.25,yscale=.1]
\foreach \x in {0.5}{
\node at (2.5,0) {$\beta$};
\node at (2.5,6) {$\alpha$};
\draw[blue] (\x,6) arc (-180:180:2);
\draw[->,blue] (\x,6) arc (-180:-90:2);
\draw[blue] (\x,0) arc (-180:0:2);
\draw[->,blue] (\x,0) arc (-180:-90:2);
\draw[blue,dashed] (\x,0) arc (180:0:2);
}
\draw[blue] (.5,0)--(.5,6);
\draw[blue] (4.5,0)--(4.5,6);
\draw[thick] (2,-2)--(2,4);
\node at (2-.5,1) {$n$};
\end{scope}
\end{tikzpicture}
}
=\sum_k
\raisebox{-1cm}{
\begin{tikzpicture}
\begin{scope}[xscale=.25,yscale=.1]
\foreach \x in {0.5}{
\node at (2.5,0) {$\beta$};
\draw[blue] (\x,6) arc (-180:180:2);
\draw[->,blue] (\x,6) arc (-180:-90:2);
\draw[blue] (\x,0) arc (-180:0:2);
\draw[->,blue] (\x,0) arc (-180:-90:2);
\draw[blue,dashed] (\x,0) arc (180:0:2);
}
\draw[blue] (.5,0)--(.5,6);
\draw[blue] (4.5,0)--(4.5,6);
\draw[thick] (2,-2)--(2,4);
\node at (2-.5,1) {$n$};
\node at (2.5,6) {$\alpha_k$};
\end{scope}

\begin{scope}[xscale=.25,yscale=.1,shift={(0,12)}]
\foreach \x in {0.5}{
\node at (2.5,6) {$\alpha$};
\draw[blue] (\x,6) arc (-180:180:2);
\draw[->,blue] (\x,6) arc (-180:-90:2);
\draw[blue] (\x,0) arc (-180:0:2);
\draw[->,blue] (\x,0) arc (-180:-90:2);
\draw[blue,dashed] (\x,0) arc (180:0:2);
}
\draw[blue] (.5,0)--(.5,6);
\draw[blue] (4.5,0)--(4.5,6);
\draw[thick] (2,-2)--(2,4);
\node at (2-.5,1) {$n$};
\node at (2.5,0) {$\beta_k$};
\end{scope}
\end{tikzpicture}}
\;.
\ee
Thus $Z(T)$ is invariant under the contraction move and it is independent of the choice the $G_T$. Therefore $Z(T)$ is well-defined for fully labelled surface tangles.

Applying the joint relation to a fully labelled tangle is equivalent to applying the inverse of the contraction move to the graph. Thus the joint relation is a relation for $Z$. Therefore we obtain an extension from $\PA_{\bullet}$ to $\SA_{\bullet}$.

\end{proof}

Consequently the general constructions of spherical planar algebras can be extended to surface algebras. For example,

\begin{corollary}\label{Cor: extension tensor}
Suppose a surface algebra $(\SA_{\bullet})_{k}$ is an extension of a planar algebra $(\PA_{\bullet})_{k}$ with sphere value $\zeta_k$, for $k=1,2$. Then
$(\SA_{\bullet})_1\otimes (\SA_{\bullet})_2$ is an extension of $(\PA_{\bullet})_1\otimes (\PA_{\bullet})_2$ with sphere value $\zeta_1\zeta_2$.
\end{corollary}

\begin{theorem}
Suppose a surface algebra $\SA_{\bullet}$ is an extension of a subfactor planar algebra $\PA_{\bullet}$ with sphere value $\zeta$. Then $\SA_{\bullet}$ is positive, if and only if $\zeta>0$.
\end{theorem}

\begin{proof}
We consider the genus-0 labelled tangle with one disc as a hemisphere.
The sphere is a composition of an unlabelled hemisphere and its mirror image, so $\zeta>0$ is necessary.
Conversely, if $\zeta>0$, then the partition function of $\SA_{\bullet}$ is positive on the sphere. By the  joint relation \eqref{Equ:joint relation}, any labelled tangles is a sum of disjoint unions of hemispheres. The positivity follows.
\end{proof}

\section{Jones-Wassermann subfactors}

Each subfactor defines a subfactor planar algebra \cite{JonPA}. A subfactor planar algebra has an alternating shading. A subfactor is called symmetrically self-dual, if its subfactor planar algebra is unshaded, see \cite{LMP17} for further discussions and examples.



The Jones-Wassermann subfactor was first studied in the framework on conformal nets \cite{LonReh95,Was98,Xu00,KawLonMug01}.
Motivated by the reconstruction program from modular tensor categories (MTC), (cf. \cite{Tur16}), to conformal field theory (CFT),
Xu and the author have constructed $m$-interval Jones-Wassermann subfactors for modular tensor categories, and proved that these subfactors are symmetrically self-dual, called the modular self-duality for MTC \cite{LiuXu}. This is a resource of a large family of unshaded planar algebras, where the input data is a modular tensor category.

We follow the notations in \cite{LiuXu}. Let $\C$ be a unitary modular tensor category and $Irr$ be the set of irreducible objects of $\C$.
For an object $X$, its dual object is denoted by $\overline{X}$. Its quantum dimension is $d(X)$.
Let $\displaystyle \mu=\sum_{X\in Irr} d(X)^2$ be the global dimension of $\C$.
Let $\SA_{\bullet}$ be the unshaded planar algebra of the $2$-interval Jones-Wassermann subfactor for $\C$, also known as the quantum double.  By parity, $\SA_n$ is zero for odd $n$ \footnote{Since we begin with unshaded planar algebras in this paper, the vector space $\SA_{2n}$ here is $\SA_{n}$ of the subfactor planar algebra in \cite{LiuXu}.}.
The 2-interval Jones-Wassermann subfactor defines a Frobenius algebra in $\C \otimes \C^{op}$ \footnote{In \cite{LiuXu}, we considered $\C \otimes \C$ instead of  $\C \otimes \C^{op}$, which is necessary in studying the $m$-interval Jones-Wassermann subfactor for all $m\geq 1$.
In this paper, we only deal with the case $m=2$. It is more convenient to work on $\C\otimes \C^{op}$.
The opposite map here corresponds to the map $\theta_2$ on $\C$ defined in \cite{LiuXu}. }.
The object $\gamma$ for the Frobenius algebra in  $\C\otimes \C^{op}$ is
\be
\gamma=\bigoplus_{X\in Irr} X \otimes X^{op}.
\ee

Since the planar algebra is unshaded, the object $\tau$ can be further decomposed as $\gamma=\tau^2$ in $\SA_{\bullet}$, where $\tau$ is the object associated with a single string. Recall that $\delta$ is the value of a closed circle in $\SA_{\bullet}$, then the Jones index $\delta^2=\mu$.

Moreover, the Hilbert space $\SA_{2n}$ is isomorphic to $\hom_{\C \otimes \C^{op}} (1,\tau^{2n})=\hom_{\C \otimes \C^{op}} (1,\gamma^{n})$.
Let $Irr^n$ be the $n^{\rm th}$ tensor power of $Irr$. Its element is given by $\vec{X}:=X_1\otimes \cdots \otimes X_{n}$.
Then $\displaystyle d(\vec{X})=\prod_{j=1}^{n} d(X_j)$.
Let $ONB(\vec{X})$ be an ortho-normal-basis of $\hom_{\C}(1,\vec{X})$.
Following the construction in \cite{LiuXu}, the partition function of the following planar diagram in $\SA_{2n}$ is given by

\be \label{Equ:spider}
\begin{tikzpicture}
\begin{scope}[scale=.8]
  \draw (.5,0) arc (180:0:.5);
  \draw (2,0) arc (180:0:.5);
  \draw (3.5,0) arc (180:0:.5);
  \draw (5,0) arc (0:180: 2.5 and 1);
\node at (2.5,0.2) {$\cdots$};
\draw[blue] (-.5,0) rectangle (5.5,1.5);
\draw[->,blue] (-.5,1.5)--(-.5,.75);
 \draw[->] (6,.7) to (7,.7);
 \node at (6.5,1) {$Z$};
 \node at (12,.5) {$\displaystyle \delta^{\frac{n}{2}} \mu_n:= \delta^{1-\frac{n}{2}} \sum_{\vec{X} \in Irr^n} d(\vec{X})^{\frac{1}{2}} \sum_{\alpha \in ONB(\vec{X})} \alpha \otimes \alpha^{op}.$};
 \end{scope}
\end{tikzpicture}
\ee

The vector space $\SA_4$ is isomorphic to $\hom_{\C \otimes \C^{op}} (1,\tau^4)$. By Frobenius reciprocity,
we can identify the Hilbert space $\SA_4$ as
\be
\hom_{\C \otimes \C^{op}} (\tau^2,\tau^2)=\bigoplus_{X\in Irr} \hom_{\C \otimes \C^{op}} (X_D, X_D) \cong L^2(Irr),
\ee
where $X_D=X \otimes X^{op}$.
It is considered as the 1-quon space for quantum information \cite{JLW-Quon}.

Take
\be \label{Equ:basis}
\beta_X=d(X)^{-1}1_{X_D},
\ee
 where $1_{X_D}$ is the identity map in $\hom_{\C \otimes \C^{op}} (X_D, X_D)$. Then
$\{ \beta_X \}_{ X \in Irr}$ form an ONB of the 1-quon space, called the {\it quantum coordinate} \cite{LiuXu}.

\begin{notation}
We denote the bra-ket notation for the 1-quon $\displaystyle \sum_{X\in Irr}  c_X \beta_X$ by $\displaystyle \sum_{X \in Irr} c_X \ket{X}$.
\end{notation}

The modular transformation $S$ of a MTC is originally defined by a hopf link.
\footnote{ The entries of the $S$ matrix is defined by the value of a Hopf link in a MTC, usually denoted by $S_{X,Y}$. Here we write it as $S_{X}^{Y}$ while considering it as a matrix on 1-quons.}
The Fourier transform on subfactors was introduced by Ocneanu in terms of paragroups \cite{Ocn88}.
In planar algebras, it turns out to be a one-string rotation of the diagram, called the string Fourier transform (SFT), denoted by $\FS$.
In general, the SFT will change the shading of diagrams in a subfactor planar algebra. It is crucial that the planar algebra $\SA_{\bullet}$ of the Jones-Wassermann subfactor is unshaded, so that the SFT is defined on each $\SA_{n}$, $n\geq 0$.
Furthermore, Xu and the author proved that the action of $\FS$ on the quantum coordinate of the 1-quon space is the $S$ matrix in \cite{LiuXu}:

\begin{proposition}\label{Prop:SFT}
On the ONB $\{ \beta_X\}_{X\in Irr}$ of $\SA_4$, the SFT $\FS$ is the modular $S$ matrix , i.e.,
\be \label{Equ:Fourier S}
\FS(\ket{X})=\sum_{Y\in irr} S_{X}^{Y}\ket{Y} .
\ee
\end{proposition}

%
%
%
%
%
%
%
%
%
%
%
%
%
%

\section{Fourier duality on 1-quons}\label{Sec:Fourier duality on 1-quons}



A quon $x$ in $\SA_4$ is represented by a labelled tangle which has one output disc with 4 points on the boundary. We modify shape of the disc as a square and represent $v$ as follows:
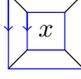
\begin{center}
\begin{tikzpicture}
\begin{scope}[xscale=.25,yscale=.25]
\draw (-2,-2)--(2,2);
\draw (-2,2)--(2,-2);
\draw[blue] (-2,-2) rectangle (2,2);
\fill[white] (-1,-1) rectangle (1,1);
\draw[blue] (-1,-1) rectangle (1,1);
\draw[->,blue] (-1,1)--(-1,0);
\draw[->,blue] (-2,2)--(-2,0);
\node at (0,0) {$x$};
\end{scope}
\end{tikzpicture}
\end{center}
The outside region belongs to the output disc, when we consider it as a genus-0 labelled tangle.

For quons $x,y \in \SA_4$, we can compose the square-like labelled tangles vertically or horizontally:

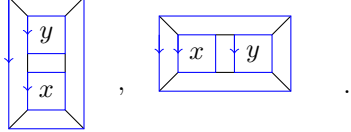
\begin{center}
\begin{tikzpicture}
\begin{scope}[xscale=.25,yscale=.25]
\draw (-2,-2)--(-1,-1);
\draw (2,-2)--(1,-1);
\draw (-1,1)--++(0,1);
\draw (1,1)--++(0,1);
\draw (-2,5)--(-1,4);
\draw (2,5)--(1,4);
\draw[blue] (-2,-2) rectangle (2,5);
\fill[white] (-1,-1) rectangle (1,1);
\draw[blue] (-1,-1) rectangle (1,1);
\fill[white] (-1,2) rectangle (1,4);
\draw[blue] (-1,2) rectangle (1,4);
\draw[->,blue] (-1,1)--(-1,0);
\draw[->,blue] (-1,4)--(-1,3);
\draw[->,blue] (-2,5)--(-2,1.5);
\node at (0,0) {$x$};
\node at (0,3) {$y$};
\end{scope}
\node at (1,0) {$,$};
\node at (4,0) {$.$};
\begin{scope}[shift={(2,.5)},rotate=-90,xscale=.25,yscale=.25]
\draw (-2,-2)--(-1,-1);
\draw (2,-2)--(1,-1);
\draw (-1,1)--++(0,1);
\draw (1,1)--++(0,1);
\draw (-2,5)--(-1,4);
\draw (2,5)--(1,4);
\draw[blue] (-2,-2) rectangle (2,5);
\fill[white] (-1,-1) rectangle (1,1);
\draw[blue] (-1,-1) rectangle (1,1);
\fill[white] (-1,2) rectangle (1,4);
\draw[blue] (-1,2) rectangle (1,4);
\draw[->,blue] (-1,-1)--(0,-1);
\draw[->,blue] (-1,2)--(0,2);
\draw[->,blue] (-1,-2)--(0,-2);
\node at (0,0) {$x$};
\node at (0,3) {$y$};
\end{scope}
\end{tikzpicture}
\end{center}
Both operations define associative multiplications on $\SA_4$.
We call the vertical composition the multiplication of $x$ and $y$, denoted by $xy$.
We call the horizontal composition the convolution of $x$ and $y$, denoted by $x*y$\footnote{The horizontal multiplication is usually called the coproduct on subfactor planar algebras.}.

Furthermore, the SFT is given by the following $90^\circ$ rotation
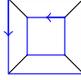
\begin{center}
\begin{tikzpicture}
\begin{scope}[xscale=.25,yscale=.25]
\draw (-2,-2)--(2,2);
\draw (-2,2)--(2,-2);
\draw[blue] (-2,-2) rectangle (2,2);
\fill[white] (-1,-1) rectangle (1,1);
\draw[blue] (-1,-1) rectangle (1,1);
\draw[->,blue] (1,1)--(0,1);
\draw[->,blue] (-2,2)--(-2,0);
\end{scope}
\end{tikzpicture}
\end{center}
It intertwines the two multiplications,
\be \label{Equ:Fourier duality}
\FS(xy)=\FS(x)*\FS(y).
\ee
This is a corner stone of the pictorial Fourier duality.

Let us consider the 1-quon space $\SA_4 \cong L^2(Irr)$ as functions on the quantum coordinates.
Then we have the following formulas for the multiplication and the convolution.
\begin{proposition}[Multiplication]
For $X,Y \in Irr$,
\be \label{Equ:multiplication}
\ket{X}\ket{Y}=\delta_{X,Y} d(X)^{-1} \ket{X}.
\ee
\end{proposition}

\begin{proof}
It follows from Equation \eqref{Equ:basis}.
\end{proof}

\begin{proposition}[Convolution]
For $X,Y \in Irr$,
\be \label{Equ:convolution}
\ket{X}*\ket{Y}=\delta^{-1}\sum_{W \in Irr }N_{X,Y}^W \ket{W},
\ee
where $N_{X,Y}^{W}=\dim\hom(W, X\otimes Y)$.
\end{proposition}

\begin{proof}
It follows from Equation \eqref{Equ:spider}.
\end{proof}

The matrix $N_X=N_{X,-}^{-}$ is called the adjacent matrix or the fusion. Verlinde first proposed that the modular transformation $S$ diagonalizes the fusion \cite{Ver88}.
The Fourier duality of 1-quons gives a conceptual explanation of this result.

\begin{theorem}[Verlinde formula]
For any $X\in Irr$,
\be \label{Equ:Verlinde formula}
\delta^{-1}SN_XS^{-1}=\sum_{Y \in Irr} S_{X}^{Y} d(Y)^{-1} \delta_Y,
\ee
where $\delta_Y$ is the projection on to $\mathbb{C}\beta_Y$.

\end{theorem}

\begin{proof}
By Equations \eqref{Equ:multiplication}, \eqref{Equ:convolution}, \eqref{Equ:Fourier duality}, \eqref{Equ:Fourier S}
\begin{align*}
\delta^{-1}SN_XS^{-1} (S\ket{W})
=&S(\ket{X}*\ket{W})\\
=&(S\ket{X})(S\ket{W})\\
=& \sum_{Y \in Irr} S_{X}^{Y} d(Y)^{-1} \delta_Y (S\ket{W}).
\end{align*}

Since $\{S\ket{W}\}_{W\in Irr}$ form an ONB of $\SA_4$, we obtain Equation \eqref{Equ:Verlinde formula}.
\end{proof}


Now we give another application of the Fourier duality on 1-quons.
The set $Irr$ of irreducible objects of $\C$ forms a fusion ring under the direct sum $\oplus$ and the tensor $\otimes$. For any subset $K \subset Irr$, we define its indicator function as
\be
P_K=\sum_{X\in K} 1_{X_D}.
\ee
Then $P_K$ is a projection in $\SA_4 \cong L^2(Irr)$. This is a bijection between subsets of $Irr$ and projections in $\SA_4$.

Let us define $SUB_{\otimes}=\{K\subset Irr | K \text{ is closed under } \otimes \}$.

\begin{theorem}\label{Thm:PK}
Take $K\subset Irr$, then $K\in SUB_{\otimes}$ iff $P_K$ is a biprojection. Consequently, if $K$ is closed under $\otimes$, the it is closed under the dual.\end{theorem}

\begin{proof}
By Equation \eqref{Equ:convolution}, if $P_K$ is a biprojection, then $K$ is closed under the tensor and the dual.
The converse statement follows from Theorem 4.12 in \cite{Liuex}.
\end{proof}

When $K\in SUB_{\otimes}$, we define $\C_K$ to be the full fusion subcategories of $\C$ whose simple objects are given by $K$.
This is a bijection between $SUB_{\otimes}$ and full fusion subcategories of $\C$.
By Theorem \ref{Thm:PK}, we obtain a bijections between full fusion subcategories $\{\C_K \}$ of $\C$ and biprojections $\{P_K\}$ of $\SA_4$.

Let $\Supp(x)$ be the trace of the support of $x$. Let $\dim \C_k$ be the global dimension of $\C_K$. Then
\be
\dim \C_K:=\sum_{X\in K} d(X)^2=\Supp(P_K).
\ee

By Theorem \ref{Thm:PK}, $P_K$ is a biprojection for $K\in SUB_{\otimes}$, so $\FS(P_K)$ is a multiple of a biprojection $P_{\hat{K}}$, for some $\hat{K}\in SUB_{\otimes}$. We call $\hat{K}$ and $P_{\hat{K}}$ the Fourier duals of $K$ and $P_{K}$ respectively. Since $\FS^2(P_K)=P_K$, the double dual is identity.
Moreover, we obtain a full subcategory $\C_{\hat{K}}$ of $\C$ that we call the Fourier dual of $\C_K$.
Then
\be
\dim \C_K \dim \C_{\hat{K}}=\Supp(P_K)\Supp(\FS(P_K))=\delta^2.  \footnote{In general, we have the Donoho-Stark uncertainty principle $\Supp(x)\Supp(\FS(x))\leq \delta^2$, see \cite{JLW16}.}
\ee

The Hausdorff-Young inequality for subfactor planar algebras has been proved in \cite{JLW16}. Applying the $\infty$-$1$ Hausdorff-Young inequality to the $S$ matrix, we recover an important inequality for unitary MTC proved by Terry Gannon \cite{Gan05}:
\begin{align*}
\|\FS(\beta_Y)\|_1 &\leq \delta^{-1} \|\beta_Y\|_{\infty} \\
\Rightarrow \left|\frac{S_X^Y}{S_{X}^0}\right| &\leq \frac{S_{0}^Y}{S_{0}^0}.
\end{align*}

The M\"{u}ger's center $C\C_{K}$ of $\C_{K}$ in $\C$ was introduced by M\"{u}ger in \cite{Mug03}.
It is defined as a full subcategory whose simple objects are given by
\be
\left\{X \in Irr \left| \quad \frac{S_X^Y}{S_{X}^0} = \frac{S_{0}^Y}{S_{0}^0},~ \forall~ Y\in K \right. \right\}.
\ee

\begin{theorem}
For $K \in SUB_{\otimes}$,
\be
\hat{K}=K=\left\{X \in Irr \left| \quad \frac{S_X^Y}{S_{X}^0} = \frac{S_{0}^Y}{S_{0}^0},~ \forall~ Y\in K \right. \right\}.
\ee
Therefore $\C_{\hat{K}}$ is M\"{u}ger's center $C\C_{K}$.
\end{theorem}

\begin{proof}
It is a special case of Theorem 4.21 in \cite{Liuex}.
\end{proof}

We summarize the results in this section in the Table \ref{Table:Fourier duality on 1-quons}.
We recover several results about M\"{u}ger's center from a different point of view.

%

%

\section{Graphic quons}

\subsection{Definitions}
In this section, we extend the unshaded subfactor planar algebra $\SA_{\bullet}$ to a surface algebra by Theorem \ref{Thm:unique extension}, still denoted by $\SA_{\bullet}$. We consider $\zeta:=Z(S_0)$ as a free variable. We study $n$-quons through the surface algebra, particularly the ones represented by surface tangles.

Recall that $\SA_4$ is the space of 1-quons. Take its $n^{\rm th}$ tensor power $(\SA_4)^n$ to be the space of $n$-quons.
Let us denote $Q_n^m:=\hom((\SA_4)^m,(\SA_4)^n)$ to be the space of transformations from $m$-quons to $n$-quons. We ignore the index when it is zero.
For $\vec{X}=X_1\otimes \cdots \otimes X_n \in Irr^n$, we define $\beta_{\vec{X}}=\beta_{X_1}\otimes \cdots \otimes \beta_{X_n}$.
Then $\{\beta_{\vec{X}}\}_{\vec{X}\in Irr^n}$ form an ONB of $Q_n$.
\begin{notation}
We denote the bra-ket notation for the n-quon $\beta_{\vec{X}}$ by $\ket{\vec{X}}$ and $\ket{\vec{X}}=\ket{X_1\cdots X_n}$. The bra-ket notation for a transformation in $Q_n^m$ is given by $\displaystyle  \sum_{\vec{Y} \in Irr^m} \sum_{\vec{X} \in Irr^n} c_{\vec{X}}^{\vec{Y}} \ket{\vec{X}}\bra{\vec{Y}}$.
\end{notation}

By the commutative diagram \eqref{Equ:reflection}, when we reverse the orientation of a disc of a surface tangle, we switch $\bra{X}$ and $\ket{X}$ in its partition function. One can consider it as the Frobenius reciprocity.
When we use the bra-ket notation for n-quons, we have an order for the tensor. Thus we also order the discs for surface tangles from 1 to $n$.
The choice of the order is identical to the action of a permutation on the tensors.

\begin{notation}
Let $LT_n^m$ be the set of labelled surface tangles with $m$ input discs and $n$ output discs, so that each disc has four boundary points.
\end{notation}
Then the partition function $Z$ is a surjective map from $LT_n^m$ to $\SA_4$.

\begin{definition}
For a genus-$g$ labelled tangle $T$ in  $LT_n^{m}$, we define the normalized quon $\ket{T}$ by
\be
\ket{T}:=Z(S_g)^{-1} Z(T).
\ee
\end{definition}

By Theorem \ref{Thm:unique extension}, the extension from spherical planar algebra to surface algebras is unique up to the choice of $\zeta=Z(S_0)$.
\begin{proposition}
The normalized quon $\ket{T}$ is independent of the choice of $\zeta$.
\end{proposition}

\begin{proof}
It follows from the joint relation \eqref{Equ:joint relation}.
\end{proof}

\begin{definition}
Let $T_n^m$ be the subset  $LT_n^m$ consisting of surface tangles.
We call $GQ_n^m:=Z(T_n^m)$ the space of graphic quon transformations and $GQ_n:=Z(T_n)$ the space of graphic n-quons.
\end{definition}

\subsection{From graphs to graphic quons}
Let $G_n$ be the set of oriented graphs on a surface whose the edges are ordered from 1 to $n$. For $\Gamma \in G_n$, let us construct a surface tangle $T_\Gamma \in T_n$: We replace each oriented edge of $\Gamma$ by an output disc with four marked points; we replace each $n$-valent vertex of $\Gamma$ by a planar diagram with $2n$ boundary points as follows,
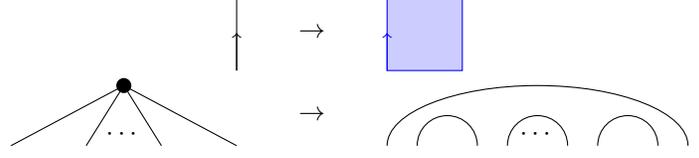
\begin{center}
\begin{tikzpicture}
\draw (0,0)--(0,1);
\draw[->] (0,0)--(0,.5);
\node at (1,.5) {$\to$};
\begin{scope}[shift={(2,0)}]
\fill[blue!20] (0,0) rectangle (1,1);
\draw[blue] (0,0) rectangle (1,1);
\draw[->,blue] (0,0)--(0,.5);
\end{scope}

\begin{scope}[shift={(-3,-1)}]
\foreach \x in {0,1,2,3}{
\draw  (1.5,.8)--(\x,0);
}
\fill (1.5,.8) circle (.1);
\node at (1.5,.16) {$\cdots$};
\node at (4,.4) {$\to$};
\begin{scope}[shift={(5,0)},scale=.8]
  \draw (.5,0) arc (180:0:.5);
  \draw (2,0) arc (180:0:.5);
  \draw (3.5,0) arc (180:0:.5);
  \draw (5,0) arc (0:180: 2.5 and 1);
\node at (2.5,0.2) {$\cdots$};
 \end{scope}
 \end{scope}
\end{tikzpicture} ~~~.
\end{center}
Moreover, we obtain a graphic quon $\ket{T_\Gamma}$.

We can also define $\ket{T_\Gamma}$ for $\Gamma$ in $\C \otimes \C^{op}$ directly: Each $k$-valent vertex of $\Gamma$ is replaced by the rotationally invariant morphism
$\delta^{\frac{k}{2}}\mu_k \in \hom_{\C\otimes \C^{op}}(1,\gamma^k)$ in Equation \eqref{Equ:spider}. Each edge is an output disc with two marked points. The target space of each output disc is $\hom(\gamma,\gamma) \cong L^2(Irr)$.

In general, when we identify the string labelled by the Frobenius algebra $\gamma$ as a pair of parallel strings, it requires a shading in the middle \footnote{This identification is a classical trick in subfactor theory. The alternating shading is essential in the study of subfactor planar algebras.}.
Here we can lift the shading by the modular self-duality of MTCs. This is used in a crucial way in \S \ref{Sec:Fourier duality}.

%

\begin{definition}
An non-zero $n$-quon is called positive, if all coefficients are non-negative.
\end{definition}

\begin{proposition}\label{Prop:positive}
For any $\Gamma \in G_n$, $\ket{T_\Gamma}$ is positive. Equivalently, its dual $\bra{T_\Gamma}$ is a positive linear functional on the tensor power of $L^2(Irr)$.
\end{proposition}

\begin{proof}
For any $\vec{X} \in Irr^{n}$, we label the $j^{\rm th}$ edge of $\Gamma$ by $\beta(X_j)=d(X)^{-1} 1_{X} \otimes 1_{X}^{op}$. We label each $k$-valent vertex of $\Gamma$ by $\delta^{\frac{k}{2}}\mu_k$. Since $\mu_k$ is a positive linear sum of $\alpha \otimes \alpha^{op}$ in Equation \eqref{Equ:spider}. It is enough to show that for each choice of $\alpha \otimes \alpha^{op}$, the value is non-negative.  For each choice, we obtain a fully surface labelled tangle in $\C$ and its opposite in $\C^{op}$. Thus the value is multiplication of a complex conjugate pair, which is non-negative. Therefore  $\bra{T_\Gamma} \vec{X} \rangle \geq 0$.
\end{proof}

If the graph $\Gamma$ is connected, then $\ket{T_{\Gamma}}$ is usually entangled for any bipartite partition.
So we call $\bra{T_\Gamma}$ a {\it topologically entangled measurement} on quons.

\begin{definition}
For $\Gamma \in G_n$, we define $\overline{\Gamma} \in G_n$ by reversing the orientations of all edges of $\Gamma$.
\end{definition}

\begin{proposition}\label{Prop:Z2}
For any $\Gamma \in G_n$,
\be
\ket{T_{\overline{\Gamma}}}=\ket{T_\Gamma}.
\ee
\end{proposition}

\begin{proof}
For any $\vec{X} \in Irr^n$,
\begin{align*}
&\langle \vec{X} \ket{T_{\overline{\Gamma}}}
=\langle \overline{\vec{X}} \ket{T_\Gamma}
=\overline{\langle \vec{X} \ket{T_\Gamma}}
=\langle \vec{X} \ket{T_\Gamma},
\end{align*}
where $\overline{\vec{X}}=\overline{X_1}\otimes \cdots \otimes \overline{X_n}$ and
the last equality follows from Proposition \ref{Prop:positive}. So $\ket{T_{\overline{\Gamma}}}=\ket{T_\Gamma}$.
\end{proof}

There are many interesting graphs on surfaces. The symmetry of (oriented) graphs leads to the symmetry of graphic quons.
For examples, there are five platonic solids on the spheres: tetrahedron, cube, octahedron, dodecahedron and icosahedron.
The number of edges are 6, 12, 12, 30, 30 respectively.

The value of the tetrahedron in $\C$ is well-known as the $6j$ symbol. Here we are working in $\C \otimes \C^{op}$, thus the value becomes the absolute square of the $6j$ symbol. If the dimension of the hom space is not 1, then we need to sum over an ONB for these $6j$ symbol squares. The sum is a good quantity to understand the global property of 6j symbols, as it is independent of the choice of the ONB.

We take an oriented tetrahedron and order the edges by 1 to 6 as shown in Fig~\ref{Fig: tetrahedron}.
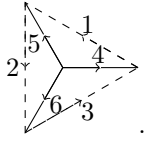
\begin{figure}\label{Fig: tetrahedron}
\begin{tikzpicture}
\coordinate (O) at (0,0);
\foreach \x in {0,1,2}
{
\coordinate (A\x) at ({cos(120*\x)},{sin(120*\x)});
\coordinate (C\x) at ({(cos(120*\x)+cos(120*(\x+1)))/2},{(sin(120*\x)+sin(120*(\x+1)))/2});
\coordinate (D\x) at ({(cos(120*\x)+cos(120*(\x+1)))/1.5},{(sin(120*\x)+sin(120*(\x+1)))/1.5});
\draw (O)--(A\x);
\draw[->,dashed] (A\x)--(C\x);
\draw[white] (.5,0) arc (0:120*\x:.5) coordinate (B\x);
\draw[->] (O)--(B\x);
\draw[white] (.5,0) arc (0:120*\x+20:.5) coordinate (E\x);
}
\node at (D0) {$1$};
\node at (D1) {$2$};
\node at (D2) {$3$};
\node at (E0) {$4$};
\node at (E1) {$5$};
\node at (E2) {$6$};

\draw[dashed] (A0) -- (A1) -- (A2) -- (A0);
\end{tikzpicture}.
\caption{Tetrahedron: The first three edges are outside and the last three edges are inside. They are order by the angle from $0^{\circ}$ to $360^{\circ}$.
The orientation of the first three edges are anti-clockwise. The orientation of the last three edges towards outside.
We consider the tetrahedron as a graph on the sphere. Dashed lines indicate that the first three edges are at the back of the sphere. }
\end{figure}
We denote this graph by $\Gamma_6$.
Then we obtain a 6-quon in terms of $6j$-symbol squares, denote by
\be\label{Equ:6j}
\ket{T_{\Gamma_6}}=
\sum_{\vec{X}\in Irr^6}
\left|{{{X_{1}X_{2}X_{3}}\choose{X_{4}X_{5}X_{6}}}}\right|^{2} \ket{\vec{X}}
\ee

For a general MTC $\C$, it could be difficult to compute the coefficients of these graphic quons. Actually the closed form of $6j$ symbols are only known for a few examples. We can manipulate these graphic quons in a pictorial way by their graphic definition, even though we do not know the algebraic closed forms of their coefficients.

\subsection{GHZ and Max}\label{Sec: GHZ Max}

Greenberg, Horne and Zeilinger introduced a multipartite resource state for quantum information, called the GHZ state, denoted by $\GHZ$ \cite{GHZ}.
In \cite{JafLiuWoz-HS}, Jaffe, Wozniakowski and the authur find another resource state following topological intuition, called $\Max$. They both  generalize the Bell state.
For the $3$-qubit case,
\begin{align*}
\GHZ&=2^{-1/2} (|000\rangle+|111\rangle), \\
\Max&=2^{-1}(|000\rangle+|011\rangle+|101\rangle+|110\rangle).
\end{align*}
We observe that $\GHZ$ and $\Max$ are Fourier duals of each other:
\be
\Max=(F\otimes F\otimes F)^{\pm1}\GHZ,
\ee
where $F$ is the discrete Fourier transform

In tensor networks, the $\GHZ$ and $\Max$ are represented as two trivalent vertices:
\begin{tikzpicture}
\begin{scope}[shift={(0,0)},xscale=-.5,yscale=.5]
\draw (0,0)--(-1,-1);
\draw (0,0)--(0,-1);
\draw (0,0)--(1,-1);
\draw[black] (0,0) circle (.2);
\fill[black] (0,0) circle (.2);
\end{scope}
\end{tikzpicture}
and
\begin{tikzpicture}
\begin{scope}[shift={(0,0)},xscale=-.5,yscale=.5]
\draw (0,0)--(-1,-1);
\draw (0,0)--(0,-1);
\draw (0,0)--(1,-1);
\fill[white] (0,0) circle (.2);
\draw[black] (0,0) circle (.2);
\end{scope}
\end{tikzpicture}.
They have been considered as two fundamental tensors in \cite{Laf03}, see also \cite{JLW-Quon,Bia17,Coe17}.

It is shown in \cite{JLW-Quon} that $\GHZ$ and $\Max$ are graphic quons, and the corresponding surface tangles are given by
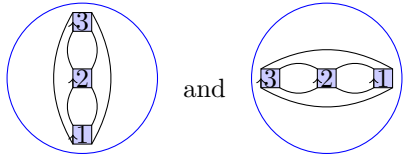
\begin{center}
\begin{tikzpicture}
\begin{scope}[scale=.25]
\draw[blue] (.5,6.5) circle (4);
\foreach \x in {0,1}{
\foreach \y in {0,1}{
\foreach \u in {0}{
\foreach \v in {1,2,3}{
\coordinate (A\u\v\x\y) at (\x+1.5*\u,\y+3*\v);
}}}}

\foreach \u in {0}{
\foreach \v in {1,2,3}{
\fill[blue!20] (A\u\v00) rectangle (A\u\v11);
\draw (A\u\v00) rectangle (A\u\v11);
\draw[->] (A\u\v00)--++(0,.5);
\node at (.5,3*\v+.5) {\v};
}}

\draw (A0101) to [bend left=30] (A0200);
\draw (A0201) to [bend left=30] (A0300);
\draw (A0301) to [bend left=-30] (A0100);

\draw (A0111) to [bend left=-30] (A0210);
\draw (A0211) to [bend left=-30] (A0310);
\draw (A0311) to [bend left=30] (A0110);

\node at (7,6) {and};
\node at (20,5) {.};

\begin{scope}[shift={(20,6)},rotate=90]
\draw[blue] (.5,6.5) circle (4);
\foreach \x in {0,1}{
\foreach \y in {0,1}{
\foreach \u in {0}{
\foreach \v in {1,2,3}{
\coordinate (A\u\v\x\y) at (\x+1.5*\u,\y+3*\v);
}}}}

\foreach \u in {0}{
\foreach \v in {1,2,3}{
\fill[blue!20] (A\u\v00) rectangle (A\u\v11);
\draw (A\u\v00) rectangle (A\u\v11);
\draw[->] (A\u\v01)--++(.5,0);
\node at (.5,3*\v+.5) {\v};
}}

\draw (A0101) to [bend left=30] (A0200);
\draw (A0201) to [bend left=30] (A0300);
\draw (A0301) to [bend left=-30] (A0100);

\draw (A0111) to [bend left=-30] (A0210);
\draw (A0211) to [bend left=-30] (A0310);
\draw (A0311) to [bend left=30] (A0110);
\end{scope}
\end{scope}
\end{tikzpicture}
\end{center}


Inspired by this observation, we generalize $\GHZ$ and $\Max$ to $n$-quons on genus-$g$ surfaces for the MTC $\C$.

\begin{definition}
Let us define the genus-$g$ tangles $GHZ_{n,g}$ and $Max_{n,g}$ in $T_n$ as follows:

\be
GHZ_{n,g}=
\raisebox{-1cm}{
\begin{tikzpicture}
\pgftransformcm{1}{0}{0}{1}{}
\begin{scope}[scale=.6]
\begin{scope}[shift={(0,0,-1)}]
\foreach \x in {1,2,3,4,5,6}
{
\draw[dashed] (2*\x,0)--++(0,1);
\draw[dashed] (2*\x+1,0)--++(0,1);
}

\foreach \x in {1,2,3}
{
\draw[dashed] (2*\x+1,0)--++(0,1);
}

\foreach \x in {1,2,3,4,5}
{
\draw[dashed] (2*\x+1,1)--++(1,0);
}
\draw[dashed] (2,0)--++(0,2);
\draw (2,2)--++(11,0)--++(0,-3);

\foreach \x in {4,5}
{
\draw[dashed] (2*\x+1,0)--++(1,0);
}
\draw[dashed] (2*4,0)--++(0,-1)--++(5,0)--++(0,1);

\end{scope}

\begin{scope}
\foreach \x in {1,2,3,4,5,6}
{
\draw (2*\x,0)--++(0,1);
\draw (2*\x+1,0)--++(0,1);
}

\foreach \x in {1,2,3}
{
\fill[blue!20]  (2*\x,0)--++(1,0)--++(0,0,-1)--++(-1,0);
\draw[blue] (2*\x,0)--++(1,0)--++(0,0,-1);
\draw[->,blue] (2*\x,0)--++(.5,0);
\draw[blue,dashed] (2*\x,0)--++(0,0,-1)--++(1,0);
\draw (2*\x+1,0)--++(0,1);
\node at (2*\x+.5,0,-.5) {\x};
}

\foreach \x in {1,2,3,4,5}
{
\draw (2*\x+1,1)--++(1,0);
}
\draw (2,0)--++(0,2)--++(11,0)--++(0,-2);

\foreach \x in {4,5}
{
\draw (2*\x+1,0)--++(1,0);
}
\draw (2*4,0)--++(0,-1)--++(5,0)--++(0,1);
\end{scope}
\end{scope}
\end{tikzpicture}
}.
\ee

\be
Max_{n,g}=
\raisebox{-1cm}{
\begin{tikzpicture}
\pgftransformcm{1}{0}{0}{1}{}
\begin{scope}[scale=.6]
\begin{scope}[shift={(0,0,-1)}]
\foreach \x in {1,2,3,4,5,6}
{
\draw[dashed] (2*\x,0)--++(0,1);
\draw[dashed] (2*\x+1,0)--++(0,1);
}

\foreach \x in {1,2,3}
{
\draw[dashed] (2*\x+1,0)--++(0,1);
}

\foreach \x in {1,2,3,4,5}
{
\draw[dashed] (2*\x+1,1)--++(1,0);
}
\draw[dashed] (2,0)--++(0,2);
\draw (2,2)--++(11,0)--++(0,-3);

\foreach \x in {4,5}
{
\draw[dashed] (2*\x+1,0)--++(1,0);
}
\draw[dashed] (2*4,0)--++(0,-1)--++(5,0)--++(0,1);

\end{scope}

\begin{scope}
\foreach \x in {1,2,3,4,5,6}
{
\draw (2*\x,0)--++(0,1);
\draw (2*\x+1,0)--++(0,1);
}

\foreach \x in {1,2,3}
{
\fill[blue!20]  (2*\x,0)--++(1,0)--++(0,0,-1)--++(-1,0);
\draw[blue] (2*\x,0)--++(1,0)--++(0,0,-1);
\draw[->,blue] (2*\x+1,0)--++(0,0,-.5);
\draw[blue,dashed] (2*\x,0)--++(0,0,-1)--++(1,0);
\draw (2*\x+1,0)--++(0,1);
\node at (2*\x+.5,0,-.5) {\x};
}

\foreach \x in {1,2,3,4,5}
{
\draw (2*\x+1,1)--++(1,0);
}
\draw (2,0)--++(0,2)--++(11,0)--++(0,-2);

\foreach \x in {4,5}
{
\draw (2*\x+1,0)--++(1,0);
}
\draw (2*4,0)--++(0,-1)--++(5,0)--++(0,1);
\end{scope}
\end{scope}
\end{tikzpicture}
}.
\ee
\end{definition}
Here we draw the tangles for $n=3$, $g=2$. The readers can figure out the general case.
The corresponding tensor network notations could be generalized (up to a scalar) as

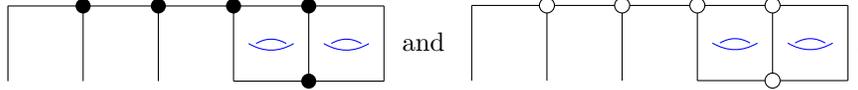
\raisebox{-.5cm}{
\begin{tikzpicture}
\foreach \x in {1,2,3,4,5,6}
{\draw (\x,0)--++(0,1);
}
\foreach \x in {2,3,4,5}
{\fill[black] (\x,1) circle (.1);
}
\foreach \x in {5}
{\fill[black] (\x,0) circle (.1);
}
\draw (1,1)--(6,1);
\draw (4,0)--(6,0);
\foreach \x in {4,5}
{
\draw[blue] (\x+.3,.5) to [bend left=30] (\x+.7,.5);
\draw[blue] (\x+.2,.5) to [bend left=-30] (\x+.8,.5);
}
\end{tikzpicture}}
\text{ and }
\raisebox{-.5cm}{
\begin{tikzpicture}
\foreach \x in {1,2,3,4,5,6}
{\draw (\x,0)--++(0,1);
}
\draw (1,1)--(6,1);
\draw (4,0)--(6,0);
\foreach \x in {2,3,4,5}
{
\fill[white] (\x,1) circle (.1);
\draw (\x,1) circle (.1);
}
\foreach \x in {5}
{
\fill[white] (\x,0) circle (.1);
\draw (\x,0) circle (.1);
}
\foreach \x in {4,5}
{
\draw[blue] (\x+.3,.5) to [bend left=30] (\x+.7,.5);
\draw[blue] (\x+.2,.5) to [bend left=-30] (\x+.8,.5);
}
\end{tikzpicture}}~~.

\begin{remark}
From tensor network to quons language, we fat a string to a cuboid. The relations of the two Frobenius algebras becomes topological isotopy in two orthogonal directions, indicated by black and white.
\end{remark}

%
%

\begin{proposition}\label{Prop:GHZ}
For $n,g \geq 0$,
\begin{align}
\GHZ_{n,g}&=  \sum_{X\in Irr} d(X)^{2-n-2g} \overbrace{|XX\cdots X \rangle}^{n \text{ entries}}.
\end{align}
\end{proposition}

\begin{proof}
By the joint relation \eqref{Equ:joint relation}, the coefficient of $\ket{\vec{X}}$ in $\GHZ_{n,g}$ is given by
\be
Z(S_0)^{-1}\sum_{\vec{Y} \in Irr^g}
\raisebox{-2cm}{
\begin{tikzpicture}
\pgftransformcm{1}{0}{0}{1}{}
\begin{scope}[scale=1]
\begin{scope}[shift={(0,0,-1)}]
\foreach \x in {1,2,3,6}
{
\draw[dashed] (2*\x,0)--++(0,1);
\draw[dashed] (2*\x+1,0)--++(0,1);
}

\foreach \x in {1,2,3}
{
\draw[dashed] (2*\x+1,0)--++(0,1);
}

\foreach \x in {1,2,3,4,5}
{
\draw[dashed] (2*\x+1,1)--++(1,0);
}
\draw[dashed] (2,0)--++(0,2);
\draw (2,2)--++(11,0)--++(0,-3);

\foreach \x in {4,5}
{
\draw[] (2*\x+1,0)--++(1,0);
}
\draw[dashed] (2*4,0)--++(0,-1)--++(5,0)--++(0,1);

\end{scope}

\begin{scope}
\foreach \x in {1,2,3,6}
{
\draw (2*\x,0)--++(0,1);
\draw (2*\x+1,0)--++(0,1);
}

\foreach \x in {1,2,3}
{
\fill[blue!20]  (2*\x,0)--++(1,0)--++(0,0,-1)--++(-1,0);
\draw[blue] (2*\x,0)--++(1,0)--++(0,0,-1);
\draw[->,blue] (2*\x+1,0)--++(-.5,0);
\draw[blue,dashed] (2*\x,0)--++(0,0,-1)--++(1,0);
\draw (2*\x+1,0)--++(0,1);
\node at (2*\x+.5,0,-.5) {$\beta_{X_\x}$};
}

\foreach \x in {4,5}
{
\fill[blue!20]  (2*\x,0)--++(1,0)--++(0,0,-1)--++(-1,0);
\draw[blue] (2*\x,0)--++(1,0)--++(0,0,-1);
\draw[->,blue] (2*\x,0)--++(.5,0);
\draw[blue] (2*\x,0)--++(0,0,-1)--++(1,0);
}
\foreach \x in {1,2}
{
\node at (2*\x+6.5,0,-.5) {$\beta_{Y_\x}$};
}

\begin{scope}[shift={(0,1)}]
\foreach \x in {4,5}
{
\fill[blue!20]  (2*\x,0)--++(1,0)--++(0,0,-1)--++(-1,0);
\draw[blue] (2*\x,0)--++(1,0);
\draw[->,blue] (2*\x+1,0)--++(-.5,0);
\draw[blue,dashed] (2*\x,0)--++(0,0,-1)--++(1,0)--++(0,0,1);
}
\foreach \x in {1,2}
{
\node at (2*\x+6.5,0,-.5) {$\beta_{Y_\x}$};
}
\end{scope}

\foreach \x in {1,2,3,4,5}
{
\draw (2*\x+1,1)--++(1,0);
}
\draw (2,0)--++(0,2)--++(11,0)--++(0,-2);

\foreach \x in {4,5}
{
\draw (2*\x+1,0)--++(1,0);
}
\draw (2*4,0)--++(0,-1)--++(5,0)--++(0,1);

\draw[dashed] (2*4,0,-1)--++(0,-1,0);
\end{scope}
\end{scope}
\end{tikzpicture}
}.
\ee
Since $\beta_{X}=d(X)^{-1}1_{X_D}$ and $1_{X_D}$ is a minimal projection, the coefficient is nonzero only when $\ket{\vec{X}}=|XX\cdots X \rangle$, for some $X\in Irr$. In this case, the coefficient is $d(X)^{2-n-2g}$.

\end{proof}

For $\vec{X}\in Irr^n$, let $\dim(\vec{X},g)$ be the dimension of the vector space consisting of vectors in genus-$g$ surface with boundary points $X_1,X_2,\ldots,X_n$ in $\C$.
Then
\begin{align*}
\dim(\vec{X},0)&=\dim\hom_{\C}(1,\vec{X}) \\
\dim(\vec{X},g)&=\sum_{\vec{Y}\in Irr^g} \dim\hom_{\C}(1,\vec{X}\otimes \vec{Y}\otimes \theta_1(\vec{Y})),\\
\end{align*}
where $\theta_1(\vec{Y})=\overline{Y_g}\otimes \cdots \otimes \overline{Y_1}$.

\begin{proposition}\label{Prop:Max}
For $n,g \geq 0$,
\begin{align}
\Max_{n,g}&=\delta^{2-n-2g} \sum_{\vec{k}\in Irr^n} \dim(\vec{X},g) |\vec{X} \rangle.
\end{align}
\end{proposition}

\begin{proof}
By the joint relation \eqref{Equ:joint relation}, the coefficient of $\ket{\vec{X}}$ in $\Max_{n,g}$ is given by

\be
Z(S_0)^{-1}\sum_{\vec{Y} \in Irr^g}
\raisebox{-2cm}{
\begin{tikzpicture}
\pgftransformcm{1}{0}{0}{1}{}
\begin{scope}[scale=1]
\begin{scope}[shift={(0,0,-1)}]
\foreach \x in {1,2,3,6}
{
\draw[dashed] (2*\x,0)--++(0,1);
\draw[dashed] (2*\x+1,0)--++(0,1);
}

\foreach \x in {1,2,3}
{
\draw[dashed] (2*\x+1,0)--++(0,1);
}

\foreach \x in {1,2,3,4,5}
{
\draw[dashed] (2*\x+1,1)--++(1,0);
}
\draw[dashed] (2,0)--++(0,2);
\draw (2,2)--++(11,0)--++(0,-3);

\foreach \x in {4,5}
{
\draw[] (2*\x+1,0)--++(1,0);
}
\draw[dashed] (2*4,0)--++(0,-1)--++(5,0)--++(0,1);

\end{scope}

\begin{scope}
\foreach \x in {1,2,3,6}
{
\draw (2*\x,0)--++(0,1);
\draw (2*\x+1,0)--++(0,1);
}

\foreach \x in {1,2,3}
{
\fill[blue!20]  (2*\x,0)--++(1,0)--++(0,0,-1)--++(-1,0);
\draw[blue] (2*\x,0)--++(1,0)--++(0,0,-1);
\draw[->,blue] (2*\x+1+1,0,-1)--++(0,0,.5);
\draw[blue,dashed] (2*\x,0)--++(0,0,-1)--++(1,0);
\draw (2*\x+1,0)--++(0,1);
\node at (2*\x+.5,0,-.5) {$\beta_{X_\x}$};
}

\foreach \x in {4,5}
{
\fill[blue!20]  (2*\x,0)--++(1,0)--++(0,0,-1)--++(-1,0);
\draw[blue] (2*\x,0)--++(1,0)--++(0,0,-1);
\draw[->,blue] (2*\x+1,0,-1)--++(0,0,.5);
\draw[blue] (2*\x,0)--++(0,0,-1)--++(1,0);
}
\foreach \x in {1,2}
{
\node at (2*\x+6.5,0,-.5) {$\beta_{Y_\x}$};
}

\begin{scope}[shift={(0,1)}]
\foreach \x in {4,5}
{
\fill[blue!20]  (2*\x,0)--++(1,0)--++(0,0,-1)--++(-1,0);
\draw[blue] (2*\x,0)--++(1,0);
\draw[->,blue] (2*\x+1,0,-1)--++(0,0,.5);
\draw[blue,dashed] (2*\x,0)--++(0,0,-1)--++(1,0)--++(0,0,1);
}
\foreach \x in {1,2}
{
\node at (2*\x+6.5,0,-.5) {$\beta_{Y_\x}$};
}
\end{scope}

\foreach \x in {1,2,3,4,5}
{
\draw (2*\x+1,1)--++(1,0);
}
\draw (2,0)--++(0,2)--++(11,0)--++(0,-2);

\foreach \x in {4,5}
{
\draw (2*\x+1,0)--++(1,0);
}
\draw (2*4,0)--++(0,-1)--++(5,0)--++(0,1);

\draw[dashed] (2*4,0,-1)--++(0,-1,0);
\end{scope}
\end{scope}
\end{tikzpicture}
}.
\ee

Appying Equation \eqref{Equ:spider}, the coefficient is
\begin{align*}
&\delta^{2-n-2g} \sum_{\vec{Y} \in Irr ^g} \sum_{\alpha_1,\alpha_2 \in ONB(\vec{X}\otimes \vec{Y} \otimes \theta_1(\vec{Y}))} \langle \alpha_1 \otimes \alpha_1^{op},  \alpha_2 \otimes \alpha_2^{op} \rangle\\
=&\delta^{2-n-2g}\sum_{\vec{Y}\in Irr^g} \dim\hom_{\C}(1,\vec{X}\otimes \vec{Y}\otimes \theta_1(\vec{Y}))\\
=&\delta^{2-n-2g} \dim(\vec{X},g).
\end{align*}

\end{proof}

\begin{definition}
Let us define the generating function for $\GHZ$ and $\Max$,
\begin{align}
\GHZ_n(z)=\sum_{g=0}^{\infty} \GHZ_{n,g} z^g,\\
\Max_n(z)=\sum_{g=0}^{\infty} \Max_{n,g} z^g.
\end{align}
\end{definition}

\begin{proposition}\label{Prop:GHZ rational}
For $n\geq0$,
\begin{align}
\GHZ_n(z)&=\sum_{k\in Irr} \frac{ d(X)^{4-n}}{ d(X)^2 -z} \overbrace{|X X\ldots X \rangle}^{n \text{ entries}}.
\end{align}
\end{proposition}

The coefficients of $\GHZ_n(z)$ are all rational functions. It is less obvious that the coefficients of $\Max_n(z)$ are also rational functions. We prove this in Theorem \ref{Thm:Maxz}

\section{Fourier duality}\label{Sec:Fourier duality}

In this section, we study the Fourier duality on graphic quons.
Without loss of generality, we only consider surface tangles in $T_n$ , i.e., all discs are output discs. Then their partition functions are graphic quons in $GQ_n$.

Recall that the SFT $\FS$ is a $90^{\circ}$ rotation of the output disc. The corresponding genus-0 tangle is given by
\be
\FS=
\raisebox{-.5cm}{
\begin{tikzpicture}
\begin{scope}[scale=1]

\fill[blue!20] (0,0+1)--(1,0+1)--(1.5,.5+1)--(.5,.5+1)--(0,0+1);
\draw[blue] (0,0+1)--(1,0+1)--(1.5,.5+1)--(.5,.5+1)--(0,0+1);
\draw[->,blue] (.5,1.5)--(.25,1.25);

\fill[blue!20] (0,0)--(1,0)--(1.5,.5)--(.5,.5)--(0,0);
\draw[blue] (0,0)--(1,0)--(1.5,.5);
\draw[blue,dashed] (1.5,.5)--(.5,.5)--(0,0);
\draw[->,blue] (0,0)--(.5,0);

\draw (0,0)--++(0,1);
\draw (1,0)--++(0,1);
\draw[dashed] (.5,.5)--++(0,1);
\draw (1.5,.5)--++(0,1);

\end{scope}
\end{tikzpicture}
}
\ee
The action of $\FS$ on the quantum coordinate $\{ \beta_X \}_{ X \in Irr}$ is identical to the $S$ matrix of $\C$.
We define the action of $\vec{\FS}$ on $T_n$ as the action of $\FS$ on all output discs. We define the action of $\vec{S}$ on $GQ_n$ as the $n^{\rm th}$ tensor power of $S$.


\begin{theorem}\label{Thm:Fourier duality}
For any unitary MTC $\C$, the following commutative diagram holds,
\begin{center}
 \begin{tikzpicture}
  \begin{scope}[node distance=4cm, auto, xscale=1,yscale=1]
  \foreach \x in {0,1,2,3} {
  \foreach \y in {0,1,2,3} {
  \coordinate (A\x\y) at ({2*\x},{.7*\y});
  }}
  \foreach \y in {0,3}{
  \node at (A0\y) {surface tangles};
  \node at (A3\y) {graphic quons};
  \draw[->] (A1\y)  to node {$Z$} (A2\y);
  }
  \draw[->] (A02) to node [swap] {$\vec{\FS}$} (A01);
  \draw[->] (A32) to node [swap] {$\vec{S}$} (A31);
  \end{scope}
  \end{tikzpicture}.
\end{center}
\end{theorem}

\begin{proof}
It follows from Proposition \ref{Prop:SFT} and \ref{Thm:unique extension}.
\end{proof}
In general, if we apply a global $90^{\circ}$ rotation to a labelled surface tangle, then its partition function is acted by the conjugation of $S$.
By Proposition \ref{Prop:Z2}, we have that
\begin{corollary}
For any oriented graph on the surface $T \in G_n$,
\be
\vec{S}^2\ket{T_\Gamma}=\ket{T_\Gamma}.
\ee
\end{corollary}
So we call the graphic quons $\ket{T_\Gamma}$ and $\vec{S}\ket{T_\Gamma}$ the Fourier dual of each other.
\begin{remark}
By Proposition \ref{Prop:positive},  the Fourier dual pair of quons are both positive. It is an interesting phenomenon that the modular transformation $S$ preserves this positivity. It is difficult to construct such positive Fourier duals algebraically.
\end{remark}

\begin{corollary}
Note that $Max_{n,g}= \vec{\FS}(GHZ_{n,g})$, for any $n,g\geq 0$, so
\be \label{Equ:MaxGHZ}
\Max_{n,g}= \vec{S}\GHZ_{n,g}.
\ee
\end{corollary}

\begin{theorem}[Verlinde formula]
For any unitary MTC $\C$ and any $n,g \geq 0$,
\begin{align}
\dim(\vec{X},g)&= \sum_{X\in Irr} (\prod_{i=1}^n S_{X_i}^{X}) (S_{X}^1)^{2-n-2g}
\end{align}

\end{theorem}

\begin{proof}
Note that $\GHZ_{n,g}$ and $\Max_{n,g}$ are computed in Propositions \ref{Prop:GHZ} and \ref{Prop:Max}, and $d(X)=\delta S_X^1$.
The statement follows from comparing the coefficients on both sides of Equation \eqref{Equ:MaxGHZ}.
\end{proof}

The higher-genus Verlinde formula was first proved by Moore and Seiberg in the framework of CFT in \cite{MooSei89}. Here we prove it for any unitary MTC as the Fourier duality of $\GHZ$ and $\Max$.  The unitary condition is not necessary in the proof.


\begin{theorem}\label{Thm:Maxz}
For any $n\geq0$,
\be
\Max_n(z)=\sum_{\vec{X}\in Irr^n} \sum_{X\in Irr} (\prod_{i=1}^n S_{X_i}^{X}) \frac{ d(X)^{4-n}}{ d(X)^2 -z} \ket{\vec{X}}.
\ee
\end{theorem}

\begin{proof}
By Equation \eqref{Equ:MaxGHZ}, we have
\be
\Max_{n}(z)= \vec{S}\GHZ_{n}(z).
\ee
By Proposition \ref{Prop:GHZ rational}, the statement holds.
\end{proof}

It is interesting that the coefficients of $\Max_n(z)$, namely the generating functions of $\{\dim(\vec{X},g)\}_{g\in \mathbb{N}}$, $\vec{X}\in Irr^n$, for all $n\geq 0$, live in a small dimensional vector space spanned by  $\left\{\displaystyle \frac{1}{ d(X)^2 -z}\right\}_{X \in Irr}$.

Note that the genus-0 $\GHZ$ and $\Max$ can be defined through the cycle graph and the dipole graph,
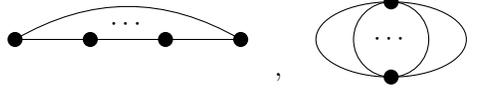
\begin{center}
\begin{tikzpicture}
\begin{scope}[shift={(-6,.5)}]
\draw (1,0)-- (4,0);
\draw (1,0) to [bend left=30] (4,0);
\foreach \x in {1,2,3,4}
{
\fill (\x,0) circle (.1);
}
\node at (2.5,.2) {$\cdots$};
\end{scope}

\node at (-1.5,0) {,};
\node at (1.5,0) {.};
\fill (0,0) circle (.1);
\fill (0,1) circle (.1);
\draw (0,0) arc (-90:90:1 and .5);
\draw (0,0) arc (-90:90:.5 and .5);
\draw (0,0) arc (270:90:.5 and .5);
\draw (0,0) arc (270:90:1 and .5);
\node at (0,.5) {$\cdots$};
\end{tikzpicture}
\end{center}
The two graphs are dual to each other.

In general, for an oriented graph $\Gamma \in G_n$ on the sphere, we obtain a genus-0 tangle $T_\Gamma$.
If we do not lift the shading, then the tangle $T_\Gamma$ has an alternating shading, and all distinguished intervals of the discs are unshaded. When we apply $\vec{\FS}$ to $T_\Gamma$, all distinguished intervals become shaded. By contracting the unshaded regions to a point, we obtain an oriented graph $\hat{\Gamma}$ in $G_n$, such that $T_{\hat{\Gamma}}=\vec{\FS}(T_{\Gamma})$. By Theorem \ref{Thm:Fourier duality}, we have that

\begin{theorem}\label{Cor:graph dual}
For any oriented graph $\Gamma \in G_n$ on the sphere,
\be \label{Equ:graph}
\ket{T_{\hat{\Gamma}}}=\vec{S}\ket{T_{\Gamma}}.
\ee
\end{theorem}

If we forget the orientation, then $\hat{\Gamma}$ is the dual graph of $\Gamma$ on the sphere.
Thus the graphic duality coincides with the Fourier duality of quons on the sphere.
However, this is not true on surfaces. One needs further assumptions for graphs on surfaces: The faces are simply connected and the edges are contractable. We call such graphs {\it local}. Then Equation \eqref{Equ:graph} remains true for local graphs.

There are interesting graphs on surfaces that are not local. Actually the graphs for $\GHZ$ and $\Max$ on higher-genus surfaces are not local.
So the quon language provides a natural extension of the graphic duality, which is compatible with the algebraic Fourier duality.

For the five platonic solids on the spheres, there are two dual pairs and one self-dual tetrahedron.
We obtain three identities for the Fourier duality. The self-duality of the tetrahedron gives the self-duality for 6j-symbols.
\begin{theorem}[6j-symbol self-duality]
For any MTC $\C$, and any $\vec{X}\in Irr^6$,
\be \label{Equ:6jRelation}
\left|{{{X_{6}~X_{5}~X_{4}}\choose{\overline{X_{3}}~\overline{X_{2}}~\overline{X_{1}}}}}\right|^{2}
= \sum_{\vec {Y}\in Irr^6} \left(\prod_{k=1}^{6}S_{X_{k}}^{Y_{k}} \right)
\left|{{{Y_{1}Y_{2}Y_{3}}\choose{Y_{4}Y_{5}Y_{6}}}}\right|^{2}
\ee
\end{theorem}

\begin{proof}
We take the tetrahedron $\Gamma_6$ in Fig~\ref{Fig: tetrahedron}.
Its dual graph $\hat{\Gamma_6}$ is given by the second. The third is isotopic to the second by $180^{\circ}$ rotation.
\begin{equation}\label{Equ:6j self dual}
\begin{tikzpicture}
\coordinate (O) at (0,0);
\foreach \x in {0,1,2}
{
\coordinate (A\x) at ({cos(120*\x)},{sin(120*\x)});
\coordinate (C\x) at ({(cos(120*\x)+cos(120*(\x+1)))/2},{(sin(120*\x)+sin(120*(\x+1)))/2});
\coordinate (D\x) at ({(cos(120*\x)+cos(120*(\x+1)))/1.5},{(sin(120*\x)+sin(120*(\x+1)))/1.5});
\draw (O)--(A\x);
\draw[->,dashed] (A\x)--(C\x);
\draw[white] (.5,0) arc (0:120*\x:.5) coordinate (B\x);
\draw[->] (O)--(B\x);
\draw[white] (.5,0) arc (0:120*\x+20:.5) coordinate (E\x);
}
\node at (D0) {$1$};
\node at (D1) {$2$};
\node at (D2) {$3$};
\node at (E0) {$4$};
\node at (E1) {$5$};
\node at (E2) {$6$};
\draw[dashed] (A0) -- (A1) -- (A2) -- (A0);
\node at (2,0) {$\to$};
\node at (6,0) {$=$};
\begin{scope}[shift={(4,0)},xscale=-1]
\coordinate (O) at (0,0);
\foreach \x in {0,1,2}
{
\coordinate (A\x) at ({cos(120*\x)},{sin(120*\x)});
\coordinate (C\x) at ({(cos(120*\x)+cos(120*(\x+1)))/2},{(sin(120*\x)+sin(120*(\x+1)))/2});
\coordinate (D\x) at ({(cos(120*\x)+cos(120*(\x+1)))/1.5},{(sin(120*\x)+sin(120*(\x+1)))/1.5});
\draw[dashed] (O)--(A\x);
\draw[->] (A\x)--(C\x);
\draw[white] (.5,0) arc (0:120*\x:.5) coordinate (B\x);
\draw[->] (A\x)--(B\x);
\draw[white] (.5,0) arc (0:120*\x+20:.5) coordinate (E\x);
}
\node at (D0) {$5$};
\node at (D1) {$4$};
\node at (D2) {$6$};
\node at (E0) {$2$};
\node at (E1) {$1$};
\node at (E2) {$3$};
\draw (A0) -- (A1) -- (A2) -- (A0);
\end{scope}
\begin{scope}[shift={(8,0)},xscale=1]
\coordinate (O) at (0,0);
\foreach \x in {0,1,2}
{
\coordinate (A\x) at ({cos(120*\x)},{sin(120*\x)});
\coordinate (C\x) at ({(cos(120*\x)+cos(120*(\x+1)))/2},{(sin(120*\x)+sin(120*(\x+1)))/2});
\coordinate (D\x) at ({(cos(120*\x)+cos(120*(\x+1)))/1.5},{(sin(120*\x)+sin(120*(\x+1)))/1.5});
\draw (O)--(A\x);
\draw[->] (A\x)--(C\x);
\draw[white] (.5,0) arc (0:120*\x:.5) coordinate (B\x);
\draw[->] (A\x)--(B\x);
\draw[white] (.5,0) arc (0:120*\x+20:.5) coordinate (E\x);
}
\node at (D0) {$6$};
\node at (D1) {$5$};
\node at (D2) {$4$};
\node at (E0) {$3$};
\node at (E1) {$2$};
\node at (E2) {$1$};
\draw[dashed] (A0) -- (A1) -- (A2) -- (A0);
\end{scope}
\end{tikzpicture}
\end{equation}

By Corollary \ref{Cor:graph dual}, we have $\ket{T_{6j}}=\vec{S}\ket{T_{\hat{6j}}}$.
Comparing the coefficients using Equation \eqref{Equ:6j}, we obtain Equation \eqref{Equ:6jRelation}.
\end{proof}

In the special case of quantum $SU(2)$, the identity for the 6j-symbol self-duality was discovered by Barrett in \cite{Bar03}, based on an interesting identity of J. Robert \cite{Rob95}. Then the identity was generalized to some other cases related to $SU(2)$ in \cite{FNR07}.

To generalize the triangle to all regular polygons, our order of edges of the tetrahedron is slightly different from Barrett's choice.
To recover Barrett's original formula, we take the following tetrahedron:

\begin{center}
\begin{tikzpicture}
\coordinate (O) at (0,0);
\foreach \x in {0,1,2}
{
\coordinate (A\x) at ({cos(120*\x)},{sin(120*\x)});
\coordinate (C\x) at ({(cos(120*\x)+cos(120*(\x+1)))/2},{(sin(120*\x)+sin(120*(\x+1)))/2});
\coordinate (D\x) at ({(cos(120*\x)+cos(120*(\x+1)))/1.5},{(sin(120*\x)+sin(120*(\x+1)))/1.5});
\draw (O)--(A\x);
\draw[->,dashed] (A\x)--(C\x);
\draw[white] (.5,0) arc (0:120*\x:.5) coordinate (B\x);
\draw[->] (O)--(B\x);
\draw[white] (.5,0) arc (0:120*\x+20:.5) coordinate (E\x);
}
\node at (D0) {$3$};
\node at (D1) {$1$};
\node at (D2) {$2$};
\node at (E0) {$4$};
\node at (E1) {$5$};
\node at (E2) {$6$};

\draw[dashed] (A0) -- (A1) -- (A2) -- (A0);

\node at (2,0) {$\to$};
\node at (6,0) {$=$};
\node at (10,0) {$\to$};

\begin{scope}[shift={(4,0)},xscale=-1]
\coordinate (O) at (0,0);
\foreach \x in {0,1,2}
{
\coordinate (A\x) at ({cos(120*\x)},{sin(120*\x)});
\coordinate (C\x) at ({(cos(120*\x)+cos(120*(\x+1)))/2},{(sin(120*\x)+sin(120*(\x+1)))/2});
\coordinate (D\x) at ({(cos(120*\x)+cos(120*(\x+1)))/1.5},{(sin(120*\x)+sin(120*(\x+1)))/1.5});

\draw[dashed] (O)--(A\x);
\draw[->] (A\x)--(C\x);
\draw[white] (.5,0) arc (0:120*\x:.5) coordinate (B\x);
\draw[->] (A\x)--(B\x);
\draw[white] (.5,0) arc (0:120*\x+20:.5) coordinate (E\x);
}

\node at (D0) {$5$};
\node at (D1) {$4$};
\node at (D2) {$6$};
\node at (E0) {$1$};
\node at (E1) {$3$};
\node at (E2) {$2$};

\draw (A0) -- (A1) -- (A2) -- (A0);
\end{scope}

\begin{scope}[shift={(8,0)},xscale=1]
\coordinate (O) at (0,0);
\foreach \x in {0,1,2}
{
\coordinate (A\x) at ({cos(120*\x)},{sin(120*\x)});
\coordinate (C\x) at ({(cos(120*\x)+cos(120*(\x+1)))/2},{(sin(120*\x)+sin(120*(\x+1)))/2});
\coordinate (D\x) at ({(cos(120*\x)+cos(120*(\x+1)))/1.5},{(sin(120*\x)+sin(120*(\x+1)))/1.5});

\draw (O)--(A\x);
\draw[->] (A\x)--(C\x);
\draw[white] (.5,0) arc (0:120*\x:.5) coordinate (B\x);
\draw[->] (A\x)--(B\x);
\draw[white] (.5,0) arc (0:120*\x+20:.5) coordinate (E\x);
}

\node at (D0) {$5$};
\node at (D1) {$4$};
\node at (D2) {$6$};
\node at (E0) {$1$};
\node at (E1) {$3$};
\node at (E2) {$2$};

\draw[dashed] (A0) -- (A1) -- (A2) -- (A0);
\end{scope}

\begin{scope}[shift={(12,0)},xscale=1]
\coordinate (O) at (0,0);
\foreach \x in {0,1,2}
{
\coordinate (A\x) at ({cos(120*\x)},{sin(120*\x)});
\coordinate (C\x) at ({(cos(120*\x)+cos(120*(\x+1)))/2},{(sin(120*\x)+sin(120*(\x+1)))/2});
\coordinate (D\x) at ({(cos(120*\x)+cos(120*(\x+1)))/1.5},{(sin(120*\x)+sin(120*(\x+1)))/1.5});

\draw (O)--(A\x);
\draw[->] (A\x)--(C\x);
\draw[white] (.5,0) arc (0:120*\x:.5) coordinate (B\x);
\draw[->] (A\x)--(B\x);
\draw[white] (.5,0) arc (0:120*\x+20:.5) coordinate (E\x);
}

\node at (D0) {$6$};
\node at (D1) {$4$};
\node at (D2) {$5$};
\node at (E0) {$1$};
\node at (E1) {$2$};
\node at (E2) {$3$};

\draw[dashed] (A0) -- (A1) -- (A2) -- (A0);
\end{scope}
\end{tikzpicture}
\end{center}
The first arrow is the graphic duality. The $=$ is a rotation. The last arrow is a vertical reflection. By Propositions \ref{Prop:positive} and \ref{Prop:Z2}, the 6-quons corresponding to the last two graphs are the same.

We can generalize the tetrahedron to a sequence of self-dual graphs on the sphere:
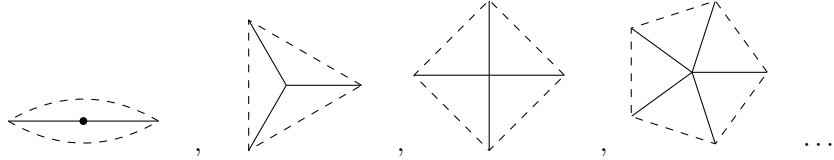
\begin{center}
\begin{tikzpicture}
\coordinate (O) at (0,0);
\foreach \x in {0,1}
{
\draw[white] (1,0) arc (0:180*\x:1) coordinate (A\x);
\fill (0,0) circle (.05);
\draw (O)--(A\x);
}
\draw[dashed] (A0) to [bend right=30] (A1)  to [bend right=30] (A0);
\end{tikzpicture}
\quad,
\begin{tikzpicture}
\coordinate (O) at (0,0);
\foreach \x in {0,1,2}
{
\draw[white] (1,0) arc (0:120*\x:1) coordinate (A\x);
\draw (O)--(A\x);
}
\draw[dashed] (A0) -- (A1) -- (A2) -- (A0);
\end{tikzpicture}
\quad,
\begin{tikzpicture}
\coordinate (O) at (0,0);
\foreach \x in {0,1,2,3}
{
\draw[white] (1,0) arc (0:360/4*\x:1) coordinate (A\x);
\draw (O)--(A\x);
}
\draw[dashed] (A0) -- (A1) -- (A2) -- (A3) -- (A0);
\end{tikzpicture}
\quad,
\begin{tikzpicture}
\coordinate (O) at (0,0);
\foreach \x in {0,1,2,3,4}
{
\draw[white] (1,0) arc (0:360/5*\x:1) coordinate (A\x);
\draw (O)--(A\x);
}
\draw[dashed] (A0) -- (A1) -- (A2) -- (A3) -- (A4) -- (A0);
\end{tikzpicture}
\quad $\cdots$
\end{center}
We order and orient the edges of each graph similar to $\Gamma_6$ in Fig~\ref{Fig: tetrahedron}, and denote the oriented graph by $\Gamma_{2n}$, for $n\geq 2$. Then we obtain a $2n$-quon, denoted by
\be
\ket{T_{\Gamma_{2n}}}=
\sum_{\vec{X}\in Irr^{2n}}
\left|{{{X_{1}\phantom{aa}X_{2}\phantom{aa} \cdots~X_{n}}\choose{X_{n+1}X_{n+2}\cdots X_{2n}}}}\right|^{2} \ket{\vec{X}}.
\ee

\begin{theorem}
For any MTC $\C$, and any $\vec{X}\in Irr^{2n}$, $n\geq 2$,
\be
\left|{{{X_{2n}~X_{2n-1}\cdots X_{n}}\choose{\overline{X_{n}} \phantom{aa} \overline{X_{n-1}} ~\cdots \overline{X_{1}}}}}\right|^{2}
= \sum_{\vec {Y}\in Irr^{2n}} \left(\prod_{k=1}^{2n}S_{X_{k}}^{Y_{k}} \right)
\left|{{{Y_{1} \phantom{aa} Y_{2} \phantom{aa} \cdots ~Y_{n}}\choose{Y_{n+1}Y_{n+2}\cdots Y_{2n}}}}\right|^{2} \;.
\ee
\end{theorem}

\begin{proof}
The dual graph of $\Gamma_{2n}$ is obtained similar to Equation \eqref{Equ:6j self dual}.
The statement follows from Corollary \ref{Cor:graph dual}.
\end{proof}

  \bibliography{bibliography}

  \bibliographystyle{amsalpha}

\end{document}